\documentclass[11pt]{amsart}
\usepackage{amssymb,amsfonts,geometry}
\usepackage{hyperref}
\usepackage{mdwlist}
\usepackage{amssymb,textcomp}
\usepackage{enumitem}
\usepackage[utf8]{inputenc}
\usepackage[T1]{fontenc}
\usepackage{tikz, tikz-cd}

\usetikzlibrary{snakes}
\usetikzlibrary{math}
\usepackage{tikz}
\usepackage{hyperref}
 \hypersetup{
     colorlinks=true,
     linktocpage=true,
     linkcolor=red,
     filecolor=blue,
     citecolor = blue,
     urlcolor=cyan,
     }
\usepackage[capitalize]{cleveref}     
\usepackage[pagewise]{lineno}%\linenumbers

\usepackage{pst-node}
\usepackage{tikz-cd} 

\usepackage{comment}
\newtheorem{lemma}{Lemma}[section]
\newtheorem{theorem}[lemma]{Theorem}
\newtheorem*{theorem*}{Theorem}
\newtheorem{claim}{Claim}[section]
\newtheorem{corollary}[lemma]{Corollary}

\newtheorem{proposition}[lemma]{Proposition}
\newtheorem*{proposition*}{Proposition}

\newtheorem*{problem*}{Problem}

\theoremstyle{definition}
\newtheorem*{claim*}{Claim}

\newtheorem{definition}{Definition}

\newtheorem{remark}{Remark}

\newcommand{\N}{{\mathbb N}}

\newcommand{\Z}{{\mathbb Z}}

\newcommand{\CA}{{\mathcal A}}
\newcommand{\CB}{{\mathcal B}}

\newcommand{\CM}{{\mathcal M}}

\newcommand{\BT}{\boldsymbol{\tau}}

\title{Asymptoticity, automorphism groups and strong orbit equivalence}
\author{Haritha Cheriyath}
\address{Centro de Modelamiento Matemático (CNRS IRL2807)\\Universidad de Chile\\Santiago, Chile}
\email{hcheriyath@cmm.uchile.cl, harithacheriyath@gmail.com}
\author{Sebasti\'an Donoso}
\address{Departamento de Ingenier\'ia Matem\'atica and Centro de Modelamiento Matemático (CNRS IRL2807)\\Universidad de Chile\\Santiago, Chile}
\email{sdonosof@uchile.cl}

\thanks{Both authors were partially funded by Centro de Modelamiento Matemático (CMM) FB210005, BASAL funds for centers of excellence from ANID-Chile. The second author was partially funded by ANID/Fondecyt/1241346.}

\subjclass[2020]{37B05, 37B10 (Primary); 37B40 (Secondary)}

\keywords{ Minimal systems, strong orbit equivalence, automorphism groups,  subshifts, asymptotic pairs, topological entropy}

\begin{document}
\begin{abstract}
   Given any strong orbit equivalence class of minimal Cantor systems and any cardinal number that is finite, countable, or the continuum, we show that there exists a minimal subshift within the given class whose number of asymptotic components is exactly the given cardinal. For finite or countable ones, we explicitly construct such examples using $\mathcal{S}$-adic subshifts. We derived the uncountable case by showing that any topological dynamical system with countably many asymptotic components has zero topological entropy. We also construct systems with arbitrarily high subexponential word complexity with only one asymptotic class. We deduce that within any strong orbit equivalence class, there exists a subshift whose automorphism group is isomorphic to $\mathbb{Z}$. 
\end{abstract}

\maketitle

\section{Introduction}
Two dynamical systems are said to be orbit equivalent if there exists a homeomorphism between the phase spaces that induces a one-to-one correspondence between their orbits. They are strong orbit equivalent if they are orbit equivalent and the cocycle map (precisely defined in \cref{sec:SOE}) is continuous everywhere except at most at one point. These definitions are weaker than topological conjugacy and define equivalence relations on the class of topological dynamical systems. In this article, as in many others in the field, we focus on minimal Cantor systems. This is because strong orbit equivalence essentially becomes topological conjugacy otherwise (see, for instance, \cite{Giordano_Putman_Skau_Topological_orbit_equiv_crossed_products:1995, Boyle_Tomimaya_bounded_orb_eq:1998}). Furthermore, in the Cantor case, we can use the machinery of Bratteli-Vershik representations to give a characterization of when two systems are strong orbit equivalent ~\cite{GPS95}.
Several works have been done on exploring the dynamical invariants that are preserved under this equivalence.
For instance, it was proved in~\cite{HPS92} that the set of invariant probability measures is preserved under strong orbit equivalence. On the other hand, topological entropy does not impose any constraints on two systems being strongly orbit equivalent. In fact, within
every strong orbit equivalence class, there exist systems with any non-negative numbers as their entropies (see~\cite{BH94,Ormes97,Sugisaki03,Sugisaki07}). Talking about eigenvalues of dynamical systems under strong orbit equivalence, it was proved that the subgroup of rational continuous eigenvalues is an invariant; however, continuous or measurable eigenvalues may vary under strong orbit equivalence~\cite{Ormes97,CDP16,DFM18}. In the case of subshifts, the complexity functions in the context of this equivalence are also studied. For instance, for any given strong orbit equivalence class, there exists a subshift with arbitrarily low superlinear complexity~\cite{CD20}.

In this article, we explore the relationship of two other conjugacy invariants within the framework of strong orbit equivalence - automorphism groups and asymptotic components. The automorphism group of a dynamical system $(X,T)$ is the group, under composition, of all homeomorphisms on $X$ that commute with the transformation $T$. They have been the subject of extensive study, especially in the context of subshifts of finite type and minimal subshifts. In the context of subshifts, the automorphism groups of a full shift~\cite{H69} or subshifts of finite type~\cite{BLR88} are shown to be profoundly large, containing copies of finite groups, countably many copies of $\Z$, free groups, among others. On the other hand, systems with low complexity exhibit a restricted automorphism group. In particular, systems with nonsuperlinear complexity have the automorphism group that is virtually $\Z$ (\cite{CK15,DDMP16}.  We refer to~\cite{CK16,S17,DDMP17} and the references therein for more work in this direction. 
%To the best of our knowledge, this article serves as the first attempt at characterizing automorphism groups under strong orbit equivalence, except for an example mentioned in~\cite{CDP16}, where the automorphism group is not invariant under orbit equivalence.
In this article, we explore, in particular, how small an automorphism group can be within a strong orbit equivalence class. It is worth mentioning that it is not always possible to find subshifts with nonsuperlinear complexity inside a strong orbit equivalence class \footnote{This is because in a class containing a subshift of nonsuperlinear complexity, the dimension group (an invariant of strong orbit equivalence) is an abelian group of finite rational rank.}. Consequently, complexity arguments alone are insufficient for obtaining small automorphism groups. Instead, we employ the concept of asymptotic components, an essential tool for studying the automorphism groups of subshifts, and examine it in detail.
%One of the main implications of our result (Corollary~\ref{cor:aut}) is that the
%automorphism groups are not invariant under strong orbit equivalence. %\textcolor{blue}{However, it is not true that strong orbit equivalence puts no restrictions at all on automorphism groups. (add)}.
Our main theorem (Theorem~\ref{thm:main1}) concludes that we can construct subshifts with asymptotic components of any cardinality - finite, countable, or uncountable, within a strong orbit equivalence class. This result uses the combinatorial structure of the subshifts to a great extent. 
The method of studying automorphism groups by exploring their asymptotic components has been seen in some earlier works, especially for subshifts with low complexity~\cite{CK15,CQY16,DDMP16,DDMP21}. Elements of the automorphism group induce permutations on the set of asymptotic components, and only the powers of $T$ fix these permutations. Therefore, a system with fewer asymptotic components will have a smaller automorphism group. This fact, together with Theorem~\ref{thm:main1}, allows us to construct a subshift whose automorphism group is isomorphic to $\Z$ within any strong orbit equivalence class. This system is immediately proven to have zero entropy using Theorem~\ref{thm:main3}, which states that positive entropies must have uncountably many asymptotic components. As a corollary to our main result, we also show that having a trivial automorphism group imposes no restrictions on its set of invariant measures (Corollary~\ref{cor:invmeas}). We also construct subshifts with smaller automorphism groups and arbitrarily low superlinear complexity or arbitrarily large subexponential complexity. 
%Since the systems with positive topological entropy have uncountably many asymptotic components, whenever we attempted to construct subshifts with smaller automorphism groups within a given strong orbit equivalence class, we always considered systems with zero entropy.
We also prove similar results for the strong orbit equivalence of Toeplitz shifts. Note that Toeplitz shifts satisfy the ``equal path number property'', which is not preserved under the constructions described in Section~\ref{Sec:finite} and Section~\ref{Sec:countable}. Consequently, we cannot guarantee that applying the same construction will eventually yield a Toeplitz shift. Therefore, we employ a different construction in this setup. Both constructions have their respective merits and demerits.

%\begin{question}
 %   Has this ever been asked if trivial automorphism group concludes anything about its invariant measures? 
%\end{question}

\subsection{Main results}
An asymptotic component consists of points that are asymptotic to a distinct point other than the elements in its orbits (the exact definition is given in Section~\ref{sec:asymclass}). 
As a preliminary result, we have the following theorem on the topological entropy of a system having at most countably many asymptotic components.

\begin{theorem}\label{thm:main3}
    Let $(X,T)$ be a topological dynamical system with at most countably many asymptotic components. Then $(X,T)$ has topological entropy zero.
\end{theorem}
We show that \cref{thm:main3} is somehow optimal, in the sense that for any given subexponential function, there exists a subshift whose complexity dominates that function along some subsequence, while possessing only a single asymptotic component (\cref{thm:subexponential}).

Now we state our main result which gives the existence of systems with finitely, countably infinite, or uncountably many asymptotic components within any strong orbit equivalence class of a minimal Cantor system. 
\begin{theorem}\label{thm:main1}
    Let $(X,T)$ be a minimal Cantor system, then there exist subshifts that are strong orbit equivalent to $(X,T)$ and have $k$-many for any $k\ge 1$, countably infinite or uncountably many asymptotic components. 
\end{theorem}
Theorem~\ref{thm:main1} is divided into three results, Theorem~\ref{thm:kmany}, Theorem~\ref{thm:count} and Corollary~\ref{cor:uncount}. Here, the case of uncountably many asymptotic components is a consequence of Theorem~\ref{thm:main3}. 

As a corollary, we also obtain subshifts with a trivial automorphism group within a strong orbit equivalence class of minimal Cantor systems.  

\begin{corollary}\label{cor:aut}
     Let $(X,T)$ be a minimal Cantor system. Then there exists a minimal subshift $(X',S)$ that is strong orbit equivalent to $(X,T)$ and $\text{Aut}(X',S)$ is generated by $S$.
\end{corollary}

There are also several other interesting corollaries of our main theorem, most of which are mentioned in Section~\ref{Sec:finite}. We also prove the following result for the strong orbit equivalence class of the Toeplitz subshift. 

\begin{theorem}\label{thm:main2}
    Let $(X,T)$ be a minimal Cantor system that is strong orbit equivalent to a Toeplitz subshift. Then there exists a Toeplitz subshift $(X',S)$ that is strong orbit equivalent to $(X,T)$ and has exactly $k$, countably, or uncountable many asymptotic components.  
\end{theorem}

\subsection{Organization of the paper}
In Section \ref{sec:prelim}, we provide a comprehensive overview of the necessary background and preliminaries. This includes an introduction to minimal Cantor systems, subshifts, Bratteli-Vershik systems and strong orbit equivalence. 
Section \ref{sec:asymclass} is devoted to the proof of Theorem \ref{thm:main3}.
In Section \ref{Sec:finite}, we focus on proving several auxiliary lemmas in order to provide the proof of Theorem~\ref{thm:main1} for the case of finitely many asymptotic components. Similar thing for the case of countably many asymptotic components is done in Section~\ref{Sec:countable}. Finally, we look at the asymptotic components of systems within the strong orbit equivalence class of the Toeplitz shift in Section~\ref{Sec:Toeplitz}.

\section{Preliminaries}\label{sec:prelim}

%\noindent Things to add
%\begin{itemize}
 %   \item Prelims on subshifts, S-adic shift, recognizability for directive morphisms
  %  \item Bratteli diagram, simple Bratteli diagram, Telescoping, Vershik map, BV theorem
   % \item Automorphism group, asymptotic class, connection between the two.
%\end{itemize}

\subsection{Minimal Cantor systems}
We say that the pair $(X,T)$ is a \emph{topological dynamical system}, or simply a {\em system}, if $X$ is a compact metric space and $T\colon X\to X$ is a homeomorphism. A point $x\in X$ is said to be \emph{periodic} if the orbit of $x$ under $T$, defined by $\text{Orb}_T(x):=\{T^n(x)\mid n\in\Z\}$, is finite. Otherwise it is said to be \emph{aperiodic}. A system $(X,T)$ is said to be \emph{aperiodic} if every point in $X$ is aperiodic and it is said to be \emph{minimal} if for every $x\in X$, $\text{Orb}_T(x)$ is dense in $X.$ 
A \emph{minimal Cantor system} is a minimal system $(X,T)$ where $X$ is a Cantor space, that is, $X$ is a non-empty totally disconnected compact metric space without isolated points.

Two dynamical systems $(X_1,T_1)$ and $(X_2,T_2)$ are said to be \emph{topologically conjugate} if there exists a homeomorphism $\phi\colon X_1\to X_2$ such that $\phi\circ T_1=T_2\circ \phi$. In such a case, $\phi$ is said to be the \emph{conjugacy map}. 
%If $\phi$ is a continuous surjective map, we say that $(X_2,T_2)$ is a \emph{factor} of $(X_1,T_1)$ and that $\phi$ is a \emph{factor map}. 

%\textcolor{blue}{Let $(X,\CB,\mu)$ be a probability space and $T:X\to X$ be measure preserving. Then, we denote by the tuple $(X,\CB,\mu,T)$, a \emph{measurable dynamical system}. For two measurable dynamical systems $(X_i,\CB_i,\mu_i,T_i)$, $i=1,2$, if there exists a measure preserving map $\pi:X_1\to X_2$ such that $\pi\circ T_1=T_2\circ\pi$, we say that $X_2$ is a \emph{factor} of $X_1$ and that $\pi$ is a \emph{measurable factor map}.}
%\textcolor{red}{measurable conjugacy and factor map}

\subsection{Symbolic dynamical systems}
 An {\em alphabet} $\CA$ consists of a finite set of symbols. A \emph{word} is a finite sequence (string) with symbols from $\CA$. The empty word is denoted by $\epsilon$. 
 If $w=w_1\dots w_n$ is a word, then $|w|=n$ denotes the length of $w$. Let $\CA^*$ denote the collection of all finite words with symbols from $\CA$, including the empty word. For two words $u=u_1\dots u_s,w=w_1\dots w_t\in\CA^*$, we define their \emph{concatenation} as $uw:=u_1\dots u_sw_1\dots w_t$.
 For two words $u,w$, we say that $u$ is a \emph{subword} of $w$, denoted as $u\prec w$, if there exist $t,v\in \CA^*$ such that $w=tuv$. When $t=\epsilon$, we say that $u$ is \emph{prefix} of $w$ and when $v=\epsilon$, we say that $u$ is a \emph{suffix} of $w$.
 
 Let $\CA^{\Z}$ denote the collection of all bi-infinite sequences with symbols from $\CA$. When $\CA$ is given the discrete topology, the space $\CA^{\Z}$ 
 is a compact metric space with respect to the product topology. 
 The metric on $\CA^\Z$ is given as follows. For $x=(x_i)_{i\in\Z},y=(y_i)_{i\in\Z}\in \CA^\Z$, $\text{dist}(x,y)=2^{-n}$ where $n=\min\{|j|:x_j\ne y_j\}$ and $\text{dist}(x,x)=0$. For $x=(x_i)_{i\in\Z}\in \CA^\Z$, we denote it as $\dots x_2x_1.x_0x_1x_2\dots$ in order to specify the zeroth coordinate.   
 Define the \emph{left shift map} $S:\CA^\Z\to \CA^\Z$ as $(S(x))_i=x_{i+1}$ where $x=(x_i)_{i\in\Z}$. A \emph{subshift} $X\subseteq \CA^\Z$ is a closed $S$-invariant subset, with the induced topology. Clearly, $(X,S)$ is a topological dynamical system. When $(X,S)$ is infinite and minimal, $X$ is a Cantor set. For a subshift, the underlying dynamics (map) is already understood to be the left shift map and hence we will only mention the map $S$ whenever necessary, otherwise we just denote a subshift to be $X$.

 Let $X$ be a subshift and $x=(x_i)_{i\in \Z}$. For $i<j$, we denote $x_{[i,j]}=x_ix_{i+1}\dots x_j$, $x_{[i,j)}=x_ix_{i+1}\dots x_{j-1}$, $x_{(i,j]}=x_{i+1}x_{i+2}\dots x_j$, $x_{[i,\infty)}=x_ix_{i+1}\dots$ and $x_{(-\infty,j]}=\dots x_{j-1}x_j$.

 A word $w\in\CA^*$ is said to be \emph{allowed} in $X$ if $w$ is a subword of a sequence in $X$, that is, there exists a sequence $x\in X$ and $i\le j$ such that $w=x_{[i,j]}$. Let $\mathcal{L}_n(X)$ denote the collection of all allowed words of length $n$ in $X$. We define $\mathcal{L}(X)=\bigcup_{n\ge 1}\mathcal{L}_n(X)$ to be the \emph{language} of $X$. Recall that the \emph{topological entropy} of a subshift $X$ is given by 
  \[
  h(X)=\lim_{n\to\infty}\frac{1}{n}\ln\#(\mathcal{L}_n(X)),
  \]
where $\#(.)$ denotes the cardinality of a set. The function $p_x:\N_0\to\N_0$ defined as $p_X(n)=\#(\mathcal{L}_n(X))$ is the complexity function of $X$. 

 For a word $w\in\mathcal{L}(X)$ and $n\in\Z$, we denote $C_{w,n}$ to be the \emph{cylinder} based at $w$ and positioned at $n$. That is,
 \[
 C_{w,n}=\{x\in X\mid x_{[n,n+|w|-1]}=w\}.
 \]
 Cylinders are clopen subsets of $X$ and they form a countable basis for the topology on $X$.
 
\subsubsection{$\mathcal{S}$-adic subshifts}
 Let $\CA,\CB$ be two finite alphabets and $\tau\colon \CA^*\to\CB^*$ be a morphism. We say that $\tau$ is \emph{primitive} if for every $a\in\CA$, all symbols $b\in \CB$ occur in $\tau(a)$ and is \emph{left (or right resp.) proper} if there exists $\ell\in \CB$ (or $r\in\CB$ resp.) such that $\tau(a)$ starts (or ends resp.) with $\ell$ (or $r$ resp.) for all $a\in \CA$. The morphism $\tau$ is said to be \emph{proper} if it is left proper and right proper. A morphism $\tau$ is a \emph{hat morphism} if for any distinct $a,b\in\CA$, the symbols appearing in $\tau(a)$ and $\tau(b)$ are all distinct. For a given morphism $\tau$, we define its incidence matrix $A$ to be a non-negative integer matrix of size $|\CA|\times|\CB|$ such that for $a\in\CA$ and $b\in\CB$, the entry $A_{a,b}$ is the number of occurrences of $b$ in $\tau(a)$. Clearly, $\tau$ is primitive if and only if $A$ is a positive matrix.

 For a sequence $(\CA_n)_{n\ge 0}$ of finite alphabets, let $\BT=(\tau_n:\CA_{n+1}^*\to \CA_n^*)_{n\ge 0}$ be a directive sequence of morphisms. For $N\ge 1$ and $0\le n<N$, define $\tau_{[n,N)}=\tau_n\circ\dots\circ\tau_{N-1}$ and $\tau_{[n,N]}=\tau_n\circ\dots\circ\tau_{N}$. We say that $\BT$ is \emph{primitive} if for every $n\ge 0$, there exists $N>n$ such that $\tau_{[n,N)}$ is primitive. 

 \noindent For $n\ge 0$, we define the \emph{$\mathcal{S}$-adic subshift generated by $\BT$  at level $n$} as
 \[
 X_{\BT}^{(n)}:=\{x\in\CA_n^\Z\mid \forall\ k\ge 0, \ x_{[-k,k]} \prec \tau_{[n,N)}(a) \text{ for some } N>n \text{ and } a\in \CA_N\}. 
 \]
 We denote by $X_{\BT}:=X_{\BT}^{(0)}$, the \emph{$\mathcal{S}$-adic subshift generated by $\BT$}. If $\BT$ is primitive, then $X_{\BT}^{(n)}$ is non-empty and minimal for all $n\ge 0$. %For more details on $\mathcal{S}$-adic subshift, refer to~\cite{Durand22_book}.

 \subsubsection{Recognizability of morphisms}
 Let $\tau\colon \CA^*\to\CB^*$ be a morphism and $Y\subseteq \CA^\Z$ be given. For $y=(y_i)_{i\in\Z}\in Y$, define $\tau(y)=\dots\tau(y_{-1}).\tau(y_0)\tau(y_1)\dots\in \CB^\Z$ by concatenation of words. %Let $S$ denote the shift map as before.
 For $x\in\CB^\Z$, if there exist $k\in\Z$ and $y\in Y$ such that $x=S^k(\tau(y))$, then we say that $(k,y)$ is a \emph{$\tau$-representation of $x$ in $Y$}. If $0\le k< |\tau(y_0)|$, we say that $(k,y)$ is a \emph{centered $\tau$-representation of $x$ in $Y$}. The morphism $\tau$ is \emph{recognizable in $Y$} if every $x\in\CB^\Z$ has at most one centered $\tau$-representation in $Y$. 

 Let $\BT=(\tau_n\colon \CA_{n+1}^*\to\CA_n^*)_{n\ge 0}$ be a directive sequence of morphisms. We say that $\BT$ is \emph{recognizable} if $\tau_n$ is recognizable in $X_{\BT}^{(n+1)}$ for each $n\ge 0$. When $\BT$ is recognizable, then for each $n\ge 0$ and $x\in X_{\BT}$, there exists a unique pair $(k,y)$ where $y\in X_{\BT}^{(n)}$ and $0\le k<|\tau_{[0,n)}(y_0)|$ such that $x=S^k(\tau_{[0,n)}(y)).$ In this case we say that $y$ is the  \emph{$\tau_{[0,n)}$- factorization} of $x$.

 We have the following lemma (\cite[Lemma 3.6]{DDMP21}) that gives the local property of recognizability. 

 \begin{lemma}\label{lemma:local}
 Let $\BT=(\tau_n)_{n\ge 0}$ be a directive sequence of morphisms. If $X_{\BT}$ is recognizable, then for all $n\ge 0$, there exists a positive integer $R$ such that whenever $x,x'\in X_{\BT}$ are such that $x_{[-R,R]}=x'_{[-R,R]}$, then $y_0=y'_0$ where $(y_i)_{i\in\Z},(y'_i)_{i\in\Z}\in X_{\BT}^{(n)}$ are the $\tau_{[0,n)}$-factorizations of $x$ and $x'$ respectively.  
\end{lemma}

 \subsubsection{Toeplitz subshifts}
 Let $\CA$ be an alphabet and $x=(x_i)_{i\in\Z}$ be an infinite sequence with symbols from $\CA$. We say that $x$ is \emph{Toeplitz} if for all $n\in \N$, there exists $p=p(n)\in\N$ such that $x_n=x_{n+kp}$ for all $k\in\Z$. A subshift $X\subseteq\CA^\Z$ is said to be \emph{Toeplitz} if it is the orbit closure under the shift map of a Toeplitz sequence. That is, $X=\overline{\{S^kx\mid k\in\Z\}}$ for some Toeplitz sequence $x$. A Toeplitz subshift is minimal. 

 \subsection{Bratteli-Vershik systems}
 In this section, we introduce Bratteli-Vershik systems, and we state a fundamental result by Herman, Putnam, and Skau, which gives the conjugacy of a minimal Cantor system with a Bratteli-Vershik system. For more details, we refer to~\cite{Durand22_book}.

 \subsubsection{Bratteli diagrams}
 A Bratteli diagram is a directed infinite graph $B=(V,E)$ with the following properties;
 \begin{itemize}
    \item Both $V$ and $E$ are partitioned into non-empty finite levels as $V=\bigcup_{n\ge 0}V_n$ and $E=\bigcup_{n\ge 1} E_n$.
    \item $V_0=\{v_0\}$ and all edges in $E_n$, for $n\ge 1$, are from $V_{n-1}$ to $V_n$.
 \end{itemize}
 For $n\ge 1$, define $s:E_n\to V_{n-1}$ and $r:E_n\to V_n$ to be the source and the range maps of edges, respectively. For $n\ge 0$, let $A_n$ denote the incidence matrix of $B$ at level $n$, that is,  $A_n$ is a non-negative integer matrix of size $|V_{n+1}|\times |V_n|$ such that $A_n(u,v)$ is the number of edges in $\{e\in E_{n+1}\mid s(e)=v \ \text{and } r(e)=u\}$. A \emph{finite path} in $B$ is given by a finite sequence of edges $(e_1,e_2,\dots,e_k)$ such that $e_i\in E_{n+i}$ for some $n\ge 0$, for all $1\le i\le k$ and $r(e_j)=s(e_{j+1})$ for all $1\le j<k$. For $k>n\ge 0$, let $E_{n,k}$ denote the collection of all paths from $V_n$ to $V_k$. Note that the number of paths from $v\in V_n$ to $u\in V_{k}$ is given by $(A_k\dots A_n)_{u,v}$. 

 For a strictly increasing subsequence $(n_k)_{k\ge 0}$ of integers with $n_0=0$, we define the \emph{telescoping of $(V,E)$ with respect to $(n_k)$} as the new Bratteli diagram $(V',E')$ where $V'_k=V_{n_k}$ and $E'_k=E_{n_{k-1}+1,n_k}$ with $s((e_1,\dots,e_\ell))=s(e_1)$ and $r((e_1,\dots,e_\ell))=r(e_\ell)$. A Bratteli diagram $(V,E)$ is said to be \emph{simple}, if there exists a telescoping $(V',E')$ of $(V,E)$ such that the corresponding incidence matrices $A'_n$ are positive matrices at each level $n\ge 0$.    

 \subsubsection{Ordered Bratteli diagrams and Vershik map}
 An ordered Bratteli diagram is a Bratteli diagram $(V,E)$ together with a partial order $\le$ on $E$ where $e,e'\in E$ are comparable if and only if $r(e)=r(e')$. 
 This induces a partial order on finite paths on $(V,E)$, called the reverse lexicographic ordering, which defines an order on any telescoping of $(V,E)$.

 We denote $X_B^{\max}$ ($X_B^{\min}$ resp.) to be the collection of $x=(e_i)_{i\ge 1}$ where $e_n$ is a maximal (minimal resp.) edge for all $n\ge 1$. 
 An ordered Bratteli diagram $B=(V,E,\le)$ is said to be \emph{properly ordered} if it is simple and $X_B^{\max}$ and $X_B^{\min}$ contain exactly one point.

 Given a properly ordered Bratteli diagram $B=(V,E,\le)$, we define a map, known as \emph{Vershik map} on 
 \[
 X_B=\{e_1e_2\dots\ \mid\ e_i\in E_{i} \text{
 and } r(e_i)=s(e_{i+1}) \text{ for } i\ge 1\}.
 \]
 We define the Vershik map $T_B:X_B\to X_B$ to be $T_B(x_{\max})=x_{\min}$ and for $x\ne x_{\max}$, the image of $x$ is obtained as follows. Let $x=(e_i)_{i\ge 1}$ and $k$ be the least integer such that $e_k$ is not a maximal edge. Let $e'_k$ be the successor of $e_k$ with respect to $\le$ where $r(e_k)=r(e'_k)$ and $(e'_1,\dots,e'_{k-1})$ be the unique minimal path from $V_0$ to $V_k$ such that $r(e'_{k-1})=s(e'_k)$. We define 
\[
T_B(x)=e'_1\dots e'_ke_{k+1}e_{k+2}\dots.
\]

 Note that $X_B$ is a closed subspace of $\prod_{i\ge 1}E_i$ and hence a compact metric space (recall that each $E_i$'s are finite). When $B=(V,E,\le)$ is simple, we have that $X_B$ is a Cantor space. Moreover, when $B$ is properly ordered, $(X_B,T_B)$ is a minimal invertible dynamical system and we call it to be the \emph{Bratteli-Vershik system}. A telescoping of Bratteli diagrams from $B=(V,E,\le)$ to $B'=(V',E',\le')$ induces a conjugacy from $(X_B,T_B)$ to $(X_{B'},T_{B'}).$ 
 Now we state the result of Herman, Putnam and Skau~\cite{HPS92} on the conjugacy between the minimal Cantor systems and Bratteli-Vershik systems. 

\begin{theorem}\label{thm:minimalCantor}
     Let $(X,T)$ be a minimal Cantor system. Then there exists a properly ordered Bratteli diagram $B=(V,E,\le)$ such that $(X_B,T_B)$ is topologically conjugate to $(X,T)$. 
\end{theorem}

\subsection{Strong orbit equivalence} \label{sec:SOE}
Let $(X_1,T_1)$ and $(X_2,T_2)$ be two topological dynamical systems. We say that $(X_1,T_1)$ is \emph{orbit equivalent} to $(X_2,T_2)$ if there exists a homeomorphism $\phi:X_1\to X_2$ such that $\phi(\text{Orb}_{T_1}(x))=\text{Orb}_{T_2}(\phi(x))$ as sets for all $x\in X_1$. Clearly, $(X_1,T_1)$ and $(X_2,T_2)$ are orbit equivalent if they are conjugate where $\phi$ can be taken to be the conjugacy map. 
The orbit equivalence between $(X_1,T_1)$ and $(X_2,T_2)$ induces two maps $\alpha,\beta:X_1\to\Z$ such that for each $x\in X_1$, 
\[
\phi\circ T_1(x)=T_2^{\alpha(x)}\circ \phi(x) \text{ and } \phi\circ T_1^{\beta(x)}(x)=T_2\circ \phi(x).
\] We call $\alpha$ and $\beta$ to be the \emph{cocycle} maps.
Continuities of $\alpha$ or $\beta$ on all points imply that $(X_1,T_1)$ is either conjugate to $(X_2,T_2)$ or to its inverse (see, for instance, \cite{Boyle_Tomimaya_bounded_orb_eq:1998}). We say that $(X_1,T_1)$ and $(X_2,T_2)$ are \emph{strong orbit equivalent} if $\alpha$ and $\beta$ have at most one point of discontinuity.
%\textcolor{blue}{Add Dimension group}

Before stating the result that characterizes minimal Cantor systems that are strong orbit equivalent in terms of their associated Bratteli-Vershik representations, we give the following definition. Two Bratteli diagrams $(V_1,E_1)$ and $(V_2,E_2)$ have \emph{common intertwining} if there exists a Bratteli diagram $(V,E)$ such that $(V_1,E_1)$ is a telescoping of $(V,E)$ to odd levels (that is, to the subsequence $(2k+1)_{k\ge 0}$) and $(V_2,E_2)$ is a telescoping of $(V,E)$ to the even levels (that is, to the subsequence $(2k)_{k\ge 0}$). Now we state the following result from~\cite{GPS95}.

\begin{proposition}\label{prop:soe}
    Let $(X_1,T_1)$ and $(X_2,T_2)$ be two minimal Cantor systems with the associated ordered Bratteli diagrams $B_1=(V_1,E_1,\le_1)$ and $B_2=(V_2,E_2,\le_2)$. Then $(X_1,T_1)$ and $(X_2,T_2)$ are strong orbit equivalent if and only if $(V_1,E_1)$ and $(V_2,E_2)$ have a common intertwining. In particular, they are strong orbit equivalent if they have the same Bratteli diagram without ordering. 
\end{proposition}
   In~\cite{GPS95}, the above result is stated in terms of the dimension groups of the associated minimal Cantor systems, the notion of which is entirely skipped from the purpose of this paper. Interested readers are referred to \cite{Durand22_book} and the references therein. 
\subsubsection{Morphisms read on an ordered Bratteli diagram}
Let $B=(V,E,\le)$ be an ordered Bratteli diagram. We define a directive sequence $\BT:=\BT_B=(\tau_n)_{n\ge 0}$ of morphisms as follows.
 
For $n\ge 0$ and $u\in V_{n+1}$, let $(e_1,\dots,e_k)$ be an ordered list of edges on $E_{n+1}$ with respect to the order $\le$ such that $\{e_1,\dots,e_k\}=r^{-1}(u)$. 
 %Also, let $(v_1,\dots,v_k)$ be an ordered list of vertices on $V_n$ such that $v_i=s(e_i)$ for $1\le i\le k$. 
For $n\ge 1$, we define the \emph{morphism read on $B$ at level $n$}, $\tau_n:V_{n+1}^*\to V_n^*$, as $\tau_n(u)=s(e_1)\dots s(e_k)$.   For $n=0$, we define $\tau_0:V_1^*\to E_1^*$ as $\tau_0(u)=e_1\dots e_k$.
 
Clearly, the incidence matrix $A_n$ of $B$ on level $n$ is the incidence matrix of the morphism $\tau_n$ for $n\ge 1$.
Moreover, for $u\in V_{n+1}$, the length of $\tau_n(u)$, denoted as $|\tau_n(u)|$, is given by the $u$-th row sum of $A_n$. Let $X_{\BT}$ denote the $\mathcal{S}$-adic subshift associated with $\BT$. We now state the following result which gives conditions that guarantee the conjugacy between $(X_{\BT},S)$ and $(X_B,T_B)$. It is a reformulation of~\cite[Proposition 4.5]{DDMP21}. 

\begin{proposition}\label{prop:conjugacy}
 % Let $B=(V,E,\le)$ be an ordered Bratteli diagram, and $\BT=(\tau_n)_{n\ge 0}$ be the directive sequence of morphisms read on $B$. Then $(X_{\BT},S)$ is conjugate to $(X_B,T_B)$ if the following holds;
 % \begin{enumerate}
 %     \item $X_{\BT}$ is minimal,
 %     \item $\tau_n$ is  proper for $n\ge 1$,
 %     \item $\tau_n$ extends by concatenation to an injective map from $X_{\BT}^{(n+1)}$ to $X_{\BT}^{(n)}$ for $n\ge 0$. 
 % \end{enumerate}
 
Let $B=(V,E,\le)$ be an ordered Bratteli diagram and $\BT=(\tau_n)_{n\ge 0}$ be the directive sequence of morphisms read on $B$. Suppose that $X_{\BT}$ is minimal, $\tau_n$ is  proper for $n\ge 1$ and each $\tau_n$ extends by concatenation to an injective map from $X_{\BT}^{(n+1)}$ to $X_{\BT}^{(n)}$ for $n\ge 0$. Then $(X_{\BT},S)$ is conjugate to $(X_B,T_B).$ 
 \end{proposition}

%Now we also discuss the recognizability of $X_{\BT}$ where $\BT$ is the directive sequence of morphisms read on an ordered Bratteli diagram $B$.  
% From~\cite[Proposition 4.6]{DDMP21} and Proposition~\ref{prop:conjugacy}, we can conclude that whenever $\BT$ is primitive and $\tau_n$ is proper for $n\ge 1$, then $X_{\BT}$ is recognizable. 

\section{Asymptotic components and automorphism  groups}\label{sec:asymclass}
In this section, we introduce the concept of asymptotic components and automorphism groups of topological dynamical systems. We describe the connection between these two notions. We also prove Theorem~\ref{thm:main3}.

\subsection{Asymptotic components and topological entropy}
Let $(X,T)$ be a topological dynamical system.
Let $\text{dist}$ denote the distance in $X$. 
Two points $x,y\in X$ are said to be \emph{left asymptotic} with respect to $T$ if $\lim_{n\to \infty}\text{dist}(T^{-n}x,T^{-n}y)=0$. Similarly, we can define the notion of right asymptotic, but we do not consider it in this paper. For simplicity, we say that two points are asymptotic if they are left asymptotic. 
We denote $x\sim y$ if $x$ and $y$ are asymptotic. 
When $X$ is a subshift, note that whenever $x\sim y$ for distinct $x,y\in X$, there exists $\ell\in \Z$ such that $x_{(-\infty,\ell)}=y_{(-\infty,\ell)}$ and $x_{\ell}\ne y_{\ell}$.
When $X$ is an infinite subshift, there are non-trivial asymptotic pairs, that is, there exist distinct $x,y\in X$ such that $x\sim y$. 

For $x,y\in X$ we say that $\text{Orb}_T(x)$ is asymptotic to $\text{Orb}_T(y)$, denoted as $x\sim_{\text{Orb}}y$, if there exist $x'\in\text{Orb}_T(x)$ and $y'\in\text{Orb}_T(y)$ that are asymptotic. This defines an equivalence relation on the collection of orbits of points on $X$ under $T$. When an equivalence class is not reduced to a single element, we call it an \emph{asymptotic component}. %\textcolor{red}{(Remove) We now state and prove the following lemma about the asymptotic pairs on the Bernoulli shift. 
%\begin{lemma}\label{lemma:bernoulli}
%    Let $(Y,S,\nu)$ be a Bernoulli shift. Then, the set $B_x:=\{y\in Y\mid y\sim x\}$ has measure zero for any $x\in Y$.
%\end{lemma}
%\begin{proof}
 %     When $\lim_{n\to\infty}\text{dist}(S^{-n}(x),S^{-n}(y))=0$, there exists $i\in\Z$ such that $x_{(-\infty,i]}=y_{(-\infty,i]}$. We partition $B_x$ as $B_x=\bigcup_{i\in\Z} B_x^i$ where $B_x^i=\{y\in Y\mid y_{(-\infty,i]}=x_{(-\infty,i]}\}.$ If $x=(x_j)_{j\in\Z},$ then $B_x^i\subseteq S^{-i+n}(C_{x_{i-n}\cdots x_{i},0})$ for all $n\ge 0$, where for a word $w$, $C_{w,0}$ denotes the cylinder based at $w$ and positioned at 0. Clearly, $\nu(B_x^i)=0$ as $\lim_{n\to \infty}\nu(C_{x_{i-n}\cdots x_{i},0})=0$.
%\end{proof}}

\noindent We now have all the necessary materials to prove Theorem~\ref{thm:main3}.

\begin{proof}[Proof of Theorem~\ref{thm:main3}]
Assume that the entropy of $(X,T)$ is positive and let $\mu$ be an ergodic Borel probability measure of $X$ such that $h(\mu)>0$. Let $X$ have at most countably many asymptotic components. Choose $x_1,x_2,\dots\in X$ to be the representatives of points on distinct asymptotic components. Let $A_i'=\{y\in X \mid y\sim x_i\}$ and $A_{i}=\{y\in X\mid y\sim_{\text{Orb}} x_i \}$ be the asymptotic and orbit asymptotic sets associated with $x_i$ for $i\in\N$. Note that $A_i=\bigcup_{n\in \Z} T^n A_i'$. 
Also,
\[
A:=\{x\in X\mid x\sim y,\ \text{for some } y\ne x\}=\bigcup_{i\in \Z} A_i\cup P,
\] where $P=\{x\in A \mid y\in\text{Orb}_T(x) \text{ whenever } x\sim y\}$. 
The elements in $P$ consist of all elements in $A$ whose orbit asymptotic class contains a single element and is not considered part of the asymptotic component.
By~\cite[Proposition 1 and Proposition 3]{BHR02}, $\mu(A)=1$. 

First, we prove that $P$ has measure zero. If $P_n=\{x\in P\ \mid\ x\sim T^nx \}$, then $P=\bigcup_{n\in\Z}P_n$. Since $P_n$ is an invariant set and $\mu$ is an ergodic measure, either $\mu(P_n)=0$ for all $n\in\Z$ or there exists $n\in\Z$ where $\mu(P_n)=1$. In the former case, we have $\mu(P)=0$ and in the later case we let $R=\{x\in X\mid \lim_{i\to\infty}T^{n_i}x=x \text{ for some subsequence } (n_i)_{i\ge 1}\}$. By Poincar\'e recurrence theorem, $\mu(R)=1$ and hence $\mu(P_n\cap R)=1$. This implies that $x=\lim_{i\to\infty}T^{n_i}x=\lim_{i\to\infty}T^{n_i+n}x=T^nx$ for almost every point in $X$. This contradicts that $h(\mu)>0$.

Hence we have $\mu\left(\bigcup_{i,n\in\Z}T^nA'_i\right)=1$, and
there exists $i\in \N$ such that  $\mu(A_{i}')>0$. Again by Poincar\'e recurrence theorem, there exists $n\geq 1$ such that $\mu(A_i'\cap T^{-n}A_i')>0$. But this is not possible as $A_i'\cap T^{-n}A_i'\subseteq P_n$.
%Set $B=\{x\in X: x \sim T^n x\}$. Then $\mu(B)>0$ and since $B$ is $T$-invariant, we obtain $\mu(B)=1$. Let $R$ be the set of recurrent points. Then $\mu(B\cap R)=1$. For $x\in B\cap R$, we have that there exists a sequence $(n_i)$ with $T^{n_ix}\to x$ and as $x\sim T^nx$, $T^{n_i+n}x\to x$ as well. Hence $T^nx=x$ for all $x\in B$. This contradicts that $h(\mu)=0$. 
\end{proof}

Using the above theorem together with the results from~\cite{Sugisaki07}, we state one part of our main result, Theorem~\ref{thm:main1}. 
\begin{corollary}\label{cor:uncount}
     Let $(X,T)$ be a minimal Cantor system. Then there exists a minimal subshift $(X',S)$ that is strong orbit equivalent to $(X,T)$ and having uncountably many asymptotic components. 
\end{corollary}
\begin{proof}
   Let $(X,T)$ be a minimal Cantor system. By~\cite[Theorem 1.1] {Sugisaki07}, there exists a minimal subshift $(X',S)$ with positive entropy that is strong orbit equivalent to $(X,T)$. This subshift $(X',S)$ must have uncountably many asymptotic components by Theorem~\ref{thm:main3}. 
\end{proof}

\subsection{Automorphism groups of a dynamical system}
 An automorphism on $(X,T)$ is a homeomorphism $\phi\colon X\to X$ such that $\phi\circ T=T\circ\phi$. Let $\text{Aut}(X,T)$ denote the group of automorphisms of $(X,T)$ under composition. Clearly, $T^n\in\text{Aut}(X,T)$ for all $n\in\Z$. We say that $\text{Aut}(X,T)$ is \emph{trivial} if it is generated by $T$.

Let  $\phi\in \text{Aut}(X,T)$. Whenever two points $x,y\in X$ are asymptotic with respect to $T$, it is easy to see that $\phi(x)$ and $\phi(y)$ are asymptotic with respect to $T$. Similarly, the orbits $\text{Orb}_T(\phi(x))$ is asymptotic to $\text{Orb}_T(\phi(y))$ whenever $\text{Orb}_T(x)$ is asymptotic to $\text{Orb}_T(y)$. Hence, $\phi\in\text{Aut}(X,T)$ defines a permutation on the set of asymptotic components.  
Using \cite[Corollary 3.3]{DDMP16}, we can easily state the following result which provides us the connection between the number of asymptotic components and automorphism group of a minimal subshift.

\begin{proposition}\label{prop:aut}
    If $(X,S)$ is a minimal  subshift with finitely many asymptotic components, then $\text{Aut}(X,S)$ virtually $\Z$. Moreover, when  $(X,S)$ has exactly one asymptotic component, then $\text{Aut}(X,S)$ is isomorphic to $\Z$.
\end{proposition}

%\begin{theorem}
 %   There exists a subshift $(X,S)$ of positive entropy such that if $\mu$ is an invariant measure of positive entropy, then there exists a subset $A\subseteq X$ of full measure such that any two points in $A$ are asymptotic to each other. 
%\end{theorem}

\section{Subshifts with finitely many asymptotic components}\label{Sec:finite}
Fix $k\ge 1$. For a given minimal Cantor system $(X,T)$, we construct a subshift that is strong orbit equivalent to $(X,T)$ and has $k$ many asymptotic components. As a corollary, we also construct a subshift within the same strong orbit equivalence class having trivial automorphism group. 

%Let $\BT=(\tau_n)_{n\ge 0}$ be a directive sequence of morphisms such that $X_{\BT}$ is recognizable. Then for $x\in X_{\BT}$, let $x^{(n)}\in X_{\BT}^{(n)}$ be such that $x=S^\ell(\tau_{[0,n)}(x^{(n)}))$ for some $0\le\ell<|\tau_{[0,n)}(x_0^{(n)})|$. For each $n\ge 0$, $x^{(n)}$ is well-defined as there exists such a unique pair $(x^{(n)},\ell)$. We have the following result about the asymptotic set of $x^{(n)}$. 

%\begin{lemma}
 %   If $\text{Orb}(x)\sim\text{Orb}(y)$ for $x,y\in X_{\BT}$, then $\text{Orb}(x^{(n)})\sim\text{Orb}(y^{(n)})$ for $n\ge 0$.
%\end{lemma}
%\begin{proof}
 %   First of all, note that if $x,x'\in X_{\BT}$ are in the same orbit, clearly, $x^{(n)},x'^{(n)}$ are in the same orbit.    Let $\text{Orb}(x)\sim\text{Orb}(y)$ for $x,y\in X_{\BT}$. Hence, without loss of generality, assume that $x_{(-\infty,0)}=x_{(-\infty,0)}$ and $x_0\ne y_0$. By the local property of recognizability, there exists $t\in \Z$ such that $x^{(n)}_{(\infty,t)}=y^{(n)}_{(\infty,t)}$. \textcolor{red}{this proof can be skipped}  
%\end{proof}

%Let $(V,E)$ be the simple Bratteli diagram and $A_n$ be the incidence matrix at level $n$. Telescoping if necessary, we assume that $(A_n)_{i,j}>0$ and $(A_n)_{i,1}>|V_{n+1}|$ for all $n\ge 1, 1\le i\le m_{n+1}$ and $1\le j\le m_n$ where $|V_k|=m_k$. Additionally, we also assume that $|V_i|>k$ for all $i\ge 1$.

%with respect to which the associated Bratteli-Vershik system is conjugate to $(X,T)$. 
\subsection{A directive sequence of morphisms} We state the following properties for a directive sequence of morphisms to be satisfied so that the $\mathcal{S}$-adic subshift generated by it will have exactly $k$ many asymptotic components.

\begin{definition}\label{def:finite-asym}
%Assume that $(A_n)_{i,j}>0$ and $(A_n)_{i,1}>3$ for all $n\ge 1, 1\le i\le m_{n+1}$ and $1\le j\le m_n$ where $|V_k|=m_k$. %This assumption can be made since $(V,E)$ is simple and we are allowed to telescope the diagram if necessary. 
    %Let $V_n$ be an alphabet of size $m_n$ (we can think of $V_n$ and $V_{n+1}$ as indexing sets of columns and rows of $A_n$, respectively). 
 Let  $\BT=(\tau_n:V^*_{n+1}\to V^*_n)_{n\ge 0}$ be a directive sequence of morphisms. We say that $\BT$    
    satisfies \emph{Property ($P_k$)} if the following conditions are true; $\tau_0$ is a hat morphism, for $n\ge 1$ and $V_n=\{v_{1,n},\dots,v_{m_n,n}\}$, we have
    \begin{enumerate}
    \item $\tau_n$ is primitive,
   %     \item $\tau_0$ is a hat morphism,
        \item for $1\le i\le k+1$,  $\tau_n(v_{i,n+1})$ has $v_{1,n}^2v_{2,n}^2\dots v_{i-1,n}^2v_{i,n}\dots v_{k,n}v_{k+1,n}$ as a prefix, 
        \item for $ i> k+1$, $\tau_n(v_{i,n+1})$ has $v_{1,n}^iv_{2,n}$ as a prefix, and 
       \item for $1\le i\le m_{n+1}$, $\tau_n(v_{i,n+1})$ has $v_{m_n,n}$ as a suffix and $\tau_n(v_{i,n+1})$ does not have $v_{m_n,n}v_{1,n}$ as a subword.
    \end{enumerate}
  
\end{definition}
Note that, in order to define Property ($P_k$), the ordering of letters in $V_n$'s are important (especially the first $k+1$ many and the last letters). Hence, whenever we talk about Property ($P_k$), we always mention that $V_n=\{v_{1,n},\dots,v_{m_n,n}\}$ with $v_{1,n},\dots,v_{k+1,n},v_{m_n,n}$ being the distinguished letters as in the definition. We state some properties of morphisms that satisfy ($P_k$) for later use, which can be easily verified. 
\begin{remark}
    Consider a directive sequence of morphisms $\BT=(\tau_n:V^*_{n+1}\to V^*_n)_{n\ge 0}$ where $V_n=\{v_{1,n},\dots,v_{m_n,n}\}$ and with the incidence matrices given by $(A_n)_{n\ge 0}$. Denote $(A_n)_{i,j}$ as the $(v_{i,n+1},v_{j,n})$-th entry of $A_n$.
    In order for $\BT$ to satisfy Property ($P_k$), the necessary requirements are, for $n\ge 1,$
\begin{enumerate}
    \item[(i)] $(A_n)_{i,j}>0$ for all $1\le j\le m_n$ and $1\le i\le m_{n+1}$,
    \item[(ii)] $m_n>k$,
    \item[(iii)] $(A_n)_{i,j}\ge 2$  for $1\le j \le i-1$ and $1\le i\le k+1,$
    \item[(iv)] $(A_n)_{i,1}\ge i$  for $i>k+1$.
\end{enumerate}
\end{remark}
\begin{remark}\label{rem:prefix}
 %Consider a directive sequence of primitive morphisms $\BT=(\tau_n:V^*_{n+1}\to V^*_n)_{n\ge 0}$ that satisfies Property ($P_k$), where $V_n=\{v_{1,n},\dots,v_{m_n,n}\}$. For $j>k$, observe that there does not exist $w\in V_n^*$ and distinct elements $u,u'\in V_{n+1},$ such that both $\tau_n(u)$ and $\tau_n(u')$ share the same prefix $wv_{j,n}$. 
 Consider a directive sequence of morphisms $\BT=(\tau_n:V^*_{n+1}\to V^*_n)_{n\ge 0}$ that satisfies Property ($P_k$), where $V_n=\{v_{1,n},\dots,v_{m_n,n}\}$. Clearly, $\tau_n$ is injective on symbols, that is, for $u, u'\in V_{n+1}$ such that $u\ne u'$, we have $\tau_n(u)\ne \tau_n(u')$. 
In fact, more can be said than just that it is injective on symbols - no two distinct images can have the same prefixes of certain type.   
 That is, there do not exist distinct elements $u,u'\in V_{n+1},$ such that both $\tau_n(u)$ and $\tau_n(u')$ share the same prefix $v_{1,n}^2\dots v_{i-1,n}^2v_{i,n}\dots v_{k+1,n}$ for some $1\le i\le k+1$. Similarly, there do not exist distinct elements $u,u'\in V_{n+1},$ such that both $\tau_n(u)$ and $\tau_n(u')$ share the same prefix $v_{1,n}^iv_{2,n}$ for some $i>k+1$. Hence, for $j>k$, there do not exist distinct elements $u,u'\in V_{n+1},$ such that both $\tau_n(u)$ and $\tau_n(u')$ share the same prefix starting with $v_{1,n}$ and ending with $v_{j,n}$.
\end{remark}
%\[
%\tau_n(u_i)=
%\begin{cases}
%v_1^2v_2^2\dots v_{i-1}^2v_{i}\dots v_kv_{k+1}\dots  v_{m_n}^{\ell_{i,m_n}}, & 1\le i\le k+1\\
 %   v_1^{i}v_2\dots v_{m_n}^{\ell_{i,m_n}}, & i>k+1
%\end{cases}
%\]
%For $n=0,$ it is a hat morphism $\tau_0:V_1^*\to E_1^*$ given by $\tau_0(v_{i,1})=e_{i,1}\dots e_{i,k_i}$ where $e_{i,j}$ is an edge from $v_{1,0}$ to $v_{i,1}$, $k_i=(A_0)_{u_i,1}$. 

%\textcolor{blue}{this is not needed here:} For a given simple Bratteli diagram $(V,E)$ with $(A_n)_{i,j}>0, (A_n)_{i,1}>|V_{n+1}|$ and $|V_n|>k$, it is possible to give a partial ordering $\ge'$ on $E$ such that $(V,E,\ge')$ is an ordered Bratteli diagram and the directive sequence of morphisms read on $(V,E,\ge')$ is primitive and satisfies Property ($P_k$). Clearly, $\tau_n$ is injective on symbols. 

%Given a Bratteli diagram $(V,E)$ and its corresponding directive sequence of morphisms $\BT$ as in Definition~\ref{def:finite-asym}, we can give a partial ordering $\ge'$ on $E$ such that $(V,E,\ge')$ is an ordered Bratteli diagram and $\BT$ is the directive sequence of morphisms read on $(V,E,\ge')$. 

\subsection{Auxiliary lemmas}
We have the following series of lemmas on the properties of the $\mathcal{S}$-adic subshift $(X_{\BT},S)$ where $\BT$ satisfies Property ($P_k$). 
 
\noindent We say that two words $v,u$ are \emph{prefix dependent} if either $u$ is a prefix of $v$ or $v$ is a prefix of $u$. 
\begin{lemma}\label{lemma:prefix_indep}
    Let $\BT=(\tau_n:V^*_{n+1}\to V^*_n)_{n\ge 0}$ be a directive sequence of morphisms that satisfy Property ($P_k$) with $V_{n}=\{v_{1,n},\dots,v_{m_n,n}\}$  for $n\ge 1$. Then for $n\ge 1$ and $1\le i\le k$, the words $\tau_{[0,n)}(v_{i,n})$ and $\tau_{[0,n)}(v_{i+1,n})$ are not prefix dependent.
\end{lemma}
\begin{proof}
We prove this by induction on $n$. Clearly, the result holds for $n=1$ as $\tau_0$ is a hat morphism.
Assume that the results hold true for $n-1$.   
    For $1\le i\le k$, we have, 
    \begin{eqnarray*}
          \tau_{[0,n)}(v_{i,n})&=&w_{i,n}\tau_{[0,n-1)}(v_{i+1,n-1})u_{i,n} \text{ and }
          \\ 
           \tau_{[0,n)}(v_{i+1,n})&=&w_{i,n}\tau_{[0,n-1)}(v_{i,n-1})u'_{i,n}
    \end{eqnarray*} where $w_{i,n}=\tau_{[0,n-1)}(v_{1,n-1}^2\dots v_{i-1,n-1}^2v_{i,n-1})$. Clearly, if $\tau_{[0,n)}(v_{i,n})$ and $\tau_{[0,n)}(v_{i+1,n})$ are prefix dependent, then $\tau_{[0,n-1)}(v_{i,n-1})$ and $\tau_{[0,n-1)}(v_{i+1,n-1})$ are prefix dependent which is a contradiction. 
\end{proof}

% \begin{definition}
%     Let $X_{\BT}$ be recognizable. For $x\in X_{\BT}^{(n+1)}$, we define the cutting point of $\tau_n(x)$  to be the subset of integers defined (refer to Figure~\ref{fig:cutting}) as,
%     \[
%     C_{\tau_n}(x):=\{-|\tau_n(x_{[-\ell,0)})|: \ell>0\}\cup\{0\}\cup \{|\tau_n(x_{[0,\ell)})|: \ell>0\}.
%     \]
% Note that when $\BT$ satisfies Property ($P_k$), the cutting points are precisely where $v_{m_n,n}$ is followed by $v_{1,n}$. 

\begin{definition}
    Let $\tau:\CA\to\CB$ be a morphism. For $x\in \CA^{\Z}$, we define the \emph{cutting point of $\tau(x)\in\CB^\Z$ by $\tau$} to be the subset of integers defined (refer to Figure~\ref{fig:cutting}) as,
    \[
    C_{\tau}(x):=\{-|\tau(x_{[-\ell,0)})|: \ell>0\}\cup\{0\}\cup \{|\tau(x_{[0,\ell)})|: \ell>0\}.
    \]

\begin{figure}[ht]
\begin{center}
\begin{tikzpicture}
\beginpgfgraphicnamed{situation-b}
\begin{scope}[very thick]
\draw[loosely dotted] (-2.5,0) -- (-1.5,0) ;
\draw (0,.15)--(0,-.15) node[above=1mm]{$c_{-1}$};
\draw (-1.5,.15)--(-1.5,-.15) node[above=1mm]{$c_{-2}$};
\draw  (-1.5,0) --  (1.5,0) ; 
\draw[snake=brace, mirror snake] (-1.5,-0.1)-- node[below] {$|\tau(x_{-2})|$} (0,-0.1) ;
\draw[snake=brace, mirror snake] (0,-0.1)-- node[below] {$|\tau(x_{-1})|$} (1.45,-0.1) ;
\draw (1.5,.15)--(1.5,-.15) node[above=1mm] {$c_{0}=0$};
\draw  (1.5,0) -- (3.4,0) ;
\draw[snake=brace, mirror snake] (1.55,-0.1)-- node[below] {$|\tau(x_{0})|$} (3.35,-0.1) ;
\draw (3.4,.15)--(3.4,-.15) node[above=1mm] {$c_{1}$};
\draw  (3.4,0) -- (4.7,0) ;
\draw[snake=brace, mirror snake] (3.45,-0.1)-- node[below] {$|\tau(x_{1})|$} (4.65,-0.1) ;
\draw (4.7,.15)--(4.7,-.15) node[above=1mm] {$c_{2}$};
\draw[loosely dotted]  (4.7,0) -- (6,0) node[right]{$\tau(x)$};
\end{scope}
\endpgfgraphicnamed
\end{tikzpicture}
\caption{Cutting points of $\tau(x)$: $C_{\tau}(x)=\{c_i\mid\ i\in\Z\}$}
\label{fig:cutting}
\end{center}
\end{figure}
\end{definition}

Note that when $\BT=(\tau_n:V_{n+1}^*\to V_n^*)_{n\ge 0}$ satisfies Property ($P_k$) with $V_n=\{v_{1,n},\dots,v_{m_n,n}\}$, the cutting points by $\tau_n$ are precisely where $v_{m_n,n}$ is followed by $v_{1,n}$. This observation, together with the fact that $\tau_n$'s are injective on symbols, it is easy to conclude that
 $X_{\BT}$ is recognizable. This notion of cutting points will be useful whenever we talk about `desubstituting' or factorizing the sequences using the recognizability of $X_{\BT}$.

\begin{definition}
Let $X$ be a subshift on an alphabet $\CA$.
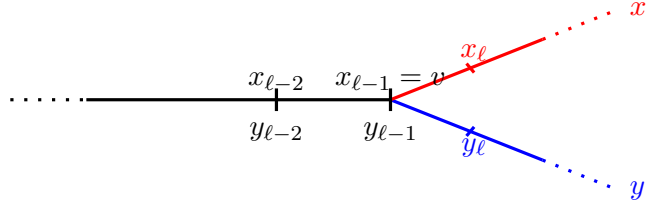
\begin{figure}[ht]
\begin{center}
\begin{tikzpicture}
\beginpgfgraphicnamed{situation-b}
\begin{scope}[very thick]
\draw[loosely dotted] (-1,0)--(0,0);
\draw (0,0)--(4,0);
\draw[blue] (4,0)--(6,-.8);
\draw[loosely dotted,blue] (6,-.8)--(7,-1.2)node[right]{$y$};
\draw[red] (4,0)--(6,.8);
\draw[loosely dotted,red] (6,.8)--(7,1.2) node[right]{$x$};
\draw (2.5,.15)--(2.5,-.15) node[above=1mm]{$x_{\ell-2}$} node[below]{$y_{\ell-2}$};
\draw (4,.15)--(4,-.15) node[above=1mm]{$x_{\ell-1}=v$} node[below]{$y_{\ell-1}$};
\draw[red] (5,.48)--(5.1,.36) node[above]{\textcolor{red}{$x_{\ell}$}} ;
\draw[blue] (5,-.48)--(5.1,-.36) node[below]{\textcolor{blue}{$y_{\ell}$}};
\end{scope}
\endpgfgraphicnamed
\end{tikzpicture}
\caption{Bifurcation point and the signal}
\label{fig:bifurcation}
\end{center}
\end{figure}
Let $\mathcal{AP}:=\mathcal{AP}(X)=\{(x,y)\mid x\sim y \text{ and } x\ne y\}$ be the collection of all non-trivial asymptotic pairs in $X$. We define 
    a function $J:\mathcal{AP}\to (\CA,\Z)$, called as a \emph{bifurcation function}, as $J(x,y)=(v,\ell)$ where $x_{(-\infty,\ell)}=y_{(-\infty,\ell)},\ x_\ell\ne y_\ell$ and $x_{\ell-1}=y_{\ell-1}=v$ (refer to Figure~\ref{fig:bifurcation}). In this case, we say that $v$ is a \emph{bifurcation signal (or simply, a signal)} for $x$. That is, $v$ is a signal for $x$ if there exist $y\in X$ and $\ell\in\Z$ such that $(x,y)\in\mathcal{AP}$ and $J(x,y)=(v,\ell)$. 
\end{definition}    
We remark that we use the notation $\mathcal{AP}$ without specifying the subshift, as this will be clear from the context. 

\begin{lemma}\label{lemma:suffix}
    Let $\BT=(\tau_n\colon V^*_{n+1}\to V^*_n)_{n\ge 0}$ be a directive sequence of  morphisms that satisfies Property ($P_k$) where $V_n=\{v_{1,n},\dots,v_{m_n,n}\}$  for $n\ge 1$.  Fix $n\ge 1$.
   If $v_{i,n}\in V_n$
   is a signal for $x^{(n)}\in X_{\BT}^{n}$, then $1\le i\le k$. \\
   Moreover, for $y^{(n)}\in X_{\BT}^{(n)}$ with $(x^{(n)}, y^{(n)})\in \mathcal{AP}$, and for $\ell\in\Z$ such that $J(x^{(n)},y^{(n)})=(v_{i,n},\ell)$, we have $\{x^{(n)}_\ell,y^{(n)}_\ell\}=\{v_{i,n},v_{i+1,n}\}$. That is, if $v_{i,n}$ is a signal, then at the bifurcation point, it is always followed by either $v_{i,n}$ or $v_{i+1,n}$.

  % We have the following to be true. 
  %   \begin{enumerate}
  %   %    \item Let $x^{(n)}\ne y^{(n)}\in X_{\BT}^{n}$ be asymptotic. Then $Sp(x^{(n)},u^{(n)})=(v_{i,n},\ell)$ for some $\ell\in\Z$ and $1\le i\le k$. 
  %   \item If $v\in V_n$ is a signal for some sequence $x^{(n)}\in X_{\BT}^{n}$, then $v=v_{i,n}$ for some $1\le i\le k$.
  %   \item Let $x^{(n)}\ne y^{(n)}\in X_{\BT}^{n}$ be asymptotic. Then for $J(x^{(n)},y^{(n)})=(v_{i,n},\ell)$ for some $\ell\in\Z$ and $1\le i\le k$, we have $\{x^{(n)}_\ell,y^{(n)}_\ell\}=\{v_{i,n},v_{i+1,n}\}$. That is, if $v_{i,n}$ is a signal for a sequence in $X_{\BT}^{(n)}$, then it is always followed by either $v_{i,n}$ or $v_{i+1,n}$.
  %   \end{enumerate}
\end{lemma}
\begin{proof}
     Let $x^{(n)}\in X_{\BT}^{(n)}$ and $v_{i,n}\in V_n$ be a signal for $x^{(n)}$. That is, there exist $y^{(n)}\in X_{\BT}^{(n)}$ and $\ell\in \Z$ such that $x^{(n)}_{(-\infty,\ell)}= y^{(n)}_{(-\infty,\ell)},\ x^{(n)}_\ell\ne y^{(n)}_\ell$ and $x^{(n)}_{\ell-1}=v_{i,n}$.
    Choose a maximum of $m<\ell-1$, such that $x^{(n)}_{[m,m+1]}=v_{m_n,n}v_{1,n}$. Hence, we have 
    \[
    x^{(n)}_{[m,\ell)}=y^{(n)}_{[m,\ell)}=v_{m_n,n}v_{1,n}\dots v_{i,n}
    \]
    having no other occurrence of $v_{m_n,n}v_{1,n}$ as a subword, but $x^{(n)}_\ell\ne y^{(n)}_\ell$. %By recognizability of $X_{\BT},$ let $x^{(n+1)},y^{(n+1)}\in X_{\BT}^{(n+1)}$ be such that $x^{(n)}=S^r(\tau_n(x^{(n+1)}))$ and $y^{(n)}=S^t(\tau_n(y^{(n+1)}))$ for some $r,t\in\Z$. 
    
    Since the cutting points by $\tau_n$ are precisely where $v_{m_n,n}$ is followed by $v_{1,n}$ and since $X_{\BT}$ is recognizable, there exist $u,u'\in V_{n+1}$ where $\tau_n(u)$ and $\tau_n(u')$ have the same prefix starting with $v_{1,n}$ and ending with $v_{i,n}$. Since $x^{(n)}_\ell\ne y^{(n)}_\ell$, we have $u\ne u'$. Hence by Remark~\ref{rem:prefix}, we have $1\le i\le k$.  
    %If $j>k$, then we can have two possible scenarios. 
  %  \begin{enumerate}
       % \item[(a)] There exists a word $w_{j,n}$ (possibly empty) that do not contain $v_{m_n,n}v_{1,n}$ as a factor and $v,v'\in V_n$ such that $v_{m_n,n}v_{1,n}^2\dots v_{i,n}^2\dots v_{k+1,n}w_{j,n}v_{j,n}v$ is a suffix of $x^{(n)}_{(-\infty,\ell]}$, and $v_{m_n,n}v_{1,n}^2\dots v_{i,n}^2 v_{i+1,n}\dots v_{k+1,n}w_{j,n}v_{j,n}v'$ is a suffix of $y^{(n)}_{(-\infty,\ell]}$, for some $1\le i\le k$. Hence, there exist $u\ne u'\in V_{n+1}$ such that $v_{1,n}^2\dots v_{i,n}^2v_{i+1,n}$ is a prefix of $\tau_{n+1}(u)$ and $v_{1,n}^2\dots v_{i,n}^2v_{i+1,n}$ is a prefix of $\tau_{n+1}(u')$.
     %  \item[(a)] 
   % \item[(b)] There exists a word $w'_{j,n}$ (possibly empty) that do not contain $v_{m_n,n}v_{1,n}$ as a factor and $v,v'\in V_n$ such that $v_{m_n,n}v_{1,n}^tw'_{j,n}v_{j,n}v$ is a suffix of $x^{(n)}_{(-\infty,\ell]}$ and $v_{m_n,n}v_{1,n}^tw'_{j,n}v_{j,n}v'$ is a suffix of $y^{(n)}_{(-\infty,\ell]}$ for some $t>k+1$: Hence, there exist $u\ne u'\in V_{n+1}$ such that  $v_{1,n}^tw'_{j,n}v_{j,n}v$ is a prefix of $\tau_{n+1}(u)$ and $v_{1,n}^tw'_{j,n}v_{j,n}v'$ is a prefix of $\tau_{n+1}(u')$. 
   % \end{enumerate}
     % Observe that neither of these scenarios is possible, as the morphisms satisfy Property ($P_k$). \\ 
%  We skip the second part of the proof, as it can be shown using the same argument. 

 For the second part of the proof, without loss of generality, let $x^{(n)}_\ell=v_{j,n}$ for $j\ne i,i+1$. As before, choose a maximum of $m<\ell-1$, such that $x^{(n)}_{[m,m+1]}=v_{m_n,n}v_{1,n}$. Then $x^{(n)}_{[m,\ell]}=v_{m_n,n}v_{1,n}\dots v_{k+1,n}\dots v_{i,n}v_{j,n}$. In particular, we have 
 \[
 x^{(n)}_{[m,\ell)}=y^{(n)}_{[m,\ell)}=v_{m_n,n}v_{1,n}\dots v_{k+1,n}\dots v_{i,n}
 \]
having no other occurrence of $v_{m_n,n}v_{1,n}$ as a subword, but $x^{(n)}_\ell\ne y^{(n)}_\ell$ which is a contradiction as before.  
\end{proof}

\begin{lemma}\label{Lemma:splittingpoint}
Let $\BT=(\tau_n)_{n\ge 0}$ be a directive sequence of morphisms that satisfies Property ($P_k$) with $V_n=\{v_{1,n},\dots,v_{m_n,n}\}$  for $n\ge 1$. 
    Let $x,y\in X_{\BT}$ be such that $x_{(-\infty,0)}=y_{(-\infty,0)}$ and $x_0\ne y_0$. For a fixed $n\ge 1$, choose $x^{(n)},y^{(n)}\in X_{\BT}^{(n)}$ and $0\le r<|\tau_{[0,n)}(x^{(n)}_0)|,\ 0\le t<|\tau_{[0,n)}(y^{(n)}_0)|$ such that $x=S^r(\tau_{[0,n)}(x^{(n)}))$ and $y=S^t(\tau_{[0,n)}(y^{(n)}))$.     
    Then the following hold;
    \begin{enumerate}
        \item $x^{(n)}\sim y^{(n)}$. In particular, $x^{(n)}_{(-\infty,0)}=y^{(n)}_{(-\infty,0)}$ and $x^{(n)}_0\ne y^{(n)}_0$. 
        \item $r=t$ and $x_{[-r,0)}=y_{[-r,0)}$ is the common prefix of $\tau_{[0,n)}(v_{i,n})$ and $\tau_{[0,n)}(v_{i+1,n})$, where $x^{(n)}_{-1}=y^{(n)}_{-1}=v_{i,n}$ for some $1\le i\le k$. 
    \end{enumerate}
\end{lemma}

\begin{proof}
 %Let $\mathcal{W}_n=\tau_{[0,n)}(V_n)$. 
    Clearly, $x^{(n)}\sim y^{(n)}$ 
    using the local property of recognizability (Lemma~\ref{lemma:local}) of $X_{\BT}$. Choose $\ell\le 0$ to be that $x^{(n)}_{(-\infty,\ell)}=y^{(n)}_{(-\infty,\ell)}$, and $x^{(n)}_\ell\ne y^{(n)}_\ell.$ By Lemma~\ref{lemma:suffix}, $x^{(n)}_{\ell-1}=y^{(n)}_{\ell-1}=v_{i,n}$ for some $1\le i\le k$ and it is followed  by either $v_{i,n}$ or $v_{i+1,n}$. Without loss of generality, assume that $x^{(n)}_\ell=v_{i,n}$ and $y^{(n)}_\ell=v_{i+1,n}$. Let $u=u_1\dots u_{k_i}=\tau_{[0,n)}(v_{i,n})$ and $w=w_1\dots w_{k_{i+1}}=\tau_{[0,n)}(v_{i+1,n})$, where $k_j=|\tau_{[0,n)}(v_{j,n})|$.
    By Lemma~\ref{lemma:prefix_indep}, $u$ and $w$ are not prefix dependent. Let $s\ge 1$ be the smallest number such that $u_s\ne w_s$. Clearly, the    
    bifurcation of $x$ and $y$ is at $x_0=u_s$ and $y_0=w_s$. That is, $J(x,y)=(u_{s-1},0)$ if $s>1$, and $J(x,y)=(u_{k_i},0)$, if $s=1$. By the definition of recognizability, we can conclude that $\ell=0$. Using this, the rest of the results are straightforward. 
\end{proof}

\begin{lemma}\label{Lemma:prefix}
Let $\BT=(\tau_n\colon V^*_{n+1}\to V^*_n)_{n\ge 0}$ be a directive sequence of  morphisms that satisfies Property ($P_k$) where $V_n=\{v_{1,n},\dots,v_{m_n,n}\}$  for $n\ge 1$.  Fix $n\ge 1$.
    %Let $X_{\BT}^{(n)}$ be the $\mathcal{S}$-adic subshift at level $n$.   
   % \textcolor{blue}{Also, assume that $X_{\BT}$ is recognizable.}
      Let $(x^{(n)}, y^{(n)})\in\mathcal{AP}(X_{\BT}^{n})$. Then we have $J(x^{(n)},y^{(n)})=(v_{i,n},\ell)$ for some $\ell\in\Z$ and $1\le i\le k$ with $x^{(n)}_{[\ell-2i,\ell)}=v_{m_n,n}v_{1,n}^2\dots v_{i-1,n}^2v_{i,n}$. That is, if $v_{i,n}$ is a signal for a sequence in $X_{\BT}^{(n)}$, then at the bifurcation point, it is always preceded by $v_{m_n,n}v_{1,n}^2\dots v_{i-1,n}^2$.
\end{lemma}
\begin{proof}
 Let $i>1$ and assume that $x^{(n)}_{[\ell-2i,\ell)}\ne v_{m_n,n}v_{1,n}^2\dots v_{i-1,n}^2v_{i,n}$. Then, as in~\ref{lemma:suffix}, choose a maximum of $m<\ell-1$, such that $x^{(n)}_{[m,m+1]}=v_{m_n,n}v_{1,n}$. We have 
 \[
 x^{(n)}_{[m,\ell)}=y^{(n)}_{[m,\ell)}=v_{m_n,n}v_{1,n}\dots v_{k+1,n}\dots v_{i,n}
 \]
having no other occurrence of $v_{m_n,n}v_{1,n}$ as a subword, but $x^{(n)}_\ell\ne y^{(n)}_\ell$, which is a contradiction as before.  

   For $i=1$, we want to prove that $x^{(n)}_{[\ell-2,\ell)}=v_{m_n,n}v_{1,n}$. Assume this is not the case. We further divide it into two cases: (1) $v_{1,n}$ is preceded by $v_{m_n,n}v_{1,n}^t$ for some $t>0$ or (2) otherwise. Note that the second case is not possible, as argued for $i>1$. But for the first case, we will find no contradiction if we argue only as above since it is possible to find two distinct $u,u'\in V_{n+1}$ such that $\tau_n(u)$ and $\tau_n(u')$ have the same prefix $v_{1,n}^{t+1}$. 
   Let us assume that this is the case. That is, $x^{(n)}_{[\ell-t-2,\ell)}=y^{(n)}_{[\ell-t-2,\ell)}=v_{m_n,n}v_{1,n}^{t+1}$ for some $t> 0$.
   By Lemma~\ref{lemma:suffix}, $v_{1,n}^{t+1}$ is followed by $v_{1,n}$ or $v_{2,n}$.     Without loss of generality, let $\tau_n(u)$ have a prefix given by $v_{1,n}^{t+2}$, and let $\tau_n(u')$ have a prefix given by $v_{1,n}^{t+1}v_{2,n}$. This implies that $u=v_{i,n+1}$ for some $i>k+1$ and $u'=v_{t+1,n+1}$.  Using Lemma~\ref{Lemma:splittingpoint} at level $n+1$, there exists a signal $v_{i,n+1}$ for some $1\le i\le k$ that is followed by $u$ and $u'$. This is not possible by Lemma~\ref{lemma:suffix}. 
\end{proof}

% \begin{lemma}\label{lemma:atmostk}
% Let $\BT=(\tau_n)_{n\ge 0}$ be the directive sequence of  morphisms that satisfies Property ($P_k$)  
% and $X_{\BT}$ be recognizable. 
%     There are at most $k$ asymptotic components in $X_{\BT}$ and each asymptotic component contains exactly two non-trivial elements.
% \end{lemma}
% \begin{proof}
   
% \end{proof}

\begin{lemma}\label{Lemma:twocomp}
    Let $\BT=(\tau_n\colon V^*_{n+1}\to V^*_n)_{n\ge 0}$ be a directive sequence of morphisms that satisfies Property ($P_k$) with $V_n=\{v_{1,n},\dots,v_{m_n,n}\}$  for $n\ge 1$. Then, each asymptotic component of $X_{\BT}$ contains exactly two elements.
\end{lemma}
\begin{proof}

If there are three non-trivial elements in the asymptotic components, without loss of generality, there exist $x,y,z\in X_{\BT}$ and $t\ge 0$ such that $x_{(-\infty,0)}=y_{(-\infty,0)}$, $x_0\ne y_0$, $x_{(-\infty,t)}=z_{(-\infty,t)}$, and $x_t\ne z_t$.
   First, we prove that $t=0$ is not possible. Otherwise, we have $y_0=z_0$ or $x_0,y_0,z_0$ to be all distinct. If $y_0=z_0$, we can replace the tuple $(x,y,z)$ with $(z,x,y)$ and continue the analysis.  For the case where $x_0,y_0,z_0$ are all distinct, we argue as follows. Let $x^{(n)},y^{(n)},z^{(n)}$ be the $\tau_{[0,n)}$-factorization of $x,y,z$, respectively in $X_{\BT}^{(n)}$. Then, by Lemma~\ref{lemma:suffix} and Lemma~\ref{Lemma:splittingpoint}, $x^{(n)}_{-1}=y^{(n)}_{-1}=z^{(n)}_{-1}=v_{i,n}$ for some $1\le i\le k$. If $x_0,y_0,z_0$ are all distinct, then $x^{(n)}_0,y^{(n)}_0,z^{(n)}_0$ are all distinct. But $v_{i,n}$ has to be followed by $v_{i,n}$ or $v_{i+1,n}$ and is therefore not possible.

   To argue the case of $t>0,$ choose $n$ large enough such that $\min\{|\tau_{[0,n)}(v)|:v\in V_n\}>t$. By Lemma~\ref{Lemma:splittingpoint}, $x^{(n)}_{-1}=y^{(n)}_{-1}$ and $x^{(n)}_0\ne y^{(n)}_0$. This implies that $x^{(n)}_{-1}=v_{i,n}$ for some $1\le i\le k$ and is now preceded by $v_{m_n,n}v_{1,n}^2\dots v_{i-1,n}^2$ by Lemma~\ref{Lemma:prefix}. Since $\min\{|\tau_{[0,n)}(v)| : v\in V_n\}>t$,
   $x^{(n)}_0=z^{(n)}_0$ and $x^{(n)}_1\ne z^{(n)}_1$. This implies that $x^{(n)}_0=v_{j,n}$ for some $1\le j\le k$ and is now preceded by $v_{m_n,n}v_{1,n}^2\dots v_{i-1,n}^2v_{i,n}$, which is a contradiction to Lemma~\ref{Lemma:prefix}.
\end{proof}

\subsection{Main results} Now we state and prove one of our main results giving a subshift of $k$ many asymptotic components. We also give some interesting corollaries of our result. 
\begin{theorem}\label{lemma:kasym}
    Let $\BT=(\tau_n\colon V^*_{n+1}\to V^*_n)_{n\ge 0}$ be a directive sequence of morphisms that satisfies Property ($P_k$)  
with $V_n=\{v_{1,n},\dots,v_{m_n,n}\}$  for $n\ge 1$. Then, there are exactly $k$ asymptotic components in $X_{\BT}$.
\end{theorem}
\begin{proof}
First, we construct $k$ many distinct non-trivial asymptotic pairs. 
For $i=1,\dots,k$, we define two decreasing nested sequences $(C_{w_{i,n},\alpha_{i,n}})_{n\ge 1}$ and $(C_{u_{i,n},\alpha_{i,n}})_{n\ge 1}$ of cylinders on $X_{\BT}$ where $w_{i,n},u_{i,n}\in\mathcal{L}(X_{\BT})$ and $\alpha_{i,n}<0$ are defined inductively. 
Let $\alpha_{i,1}=-|\tau_0(v_{1,1}^2\dots v_{i-1,1}^2v_{i,1})|$ and for $n>1,$
\[\alpha_{i,n}=-|\tau_{[0,n)}(v_{1,n}^2\dots v_{i-1,n}^2v_{i,n})|+\alpha_{i,n-1}.\] 
For $n\ge 1$, define 
\begin{equation*}
   w_{i,n}=  \begin{cases}
      \tau_{[0,n)}(v_{1,n}^2\dots v_{i-1,n}^2v_{i,n}^2), &\text{ if $n$ is odd},\\
      \tau_{[0,n)}(v_{1,n}^2\dots v_{i-1,n}^2v_{i,n}v_{i+1,n}),& \text{ if $n$ is even},
    \end{cases}
\end{equation*}
\begin{equation*}
   u_{i,n}=  \begin{cases}
       \tau_{[0,n)}(v_{1,n}^2\dots v_{i-1,n}^2v_{i,n}v_{i+1,n}),
       &\text{ if $n$ is odd},\\
      \tau_{[0,n)}(v_{1,n}^2\dots v_{i-1,n}^2v_{i,n}^2), & \text{ if $n$ is even}.
    \end{cases}
\end{equation*}
% \begin{eqnarray*}
%     w_{i,2n-1}&=&\tau_{[0,2n-1)}(v_{1,2n-1}^2\dots v_{i,2n-1}^2),\\ w_{i,2n}&=&\tau_{[0,2n)}(v_{1,2n}^2\dots v_{i,2n}v_{i+1,2n}),\\
%     u_{i,2n-1}&=&\tau_{[0,2n-1)}(v_{1,2n-1}^2\dots v_{i,2n-1}v_{i+1,2n-1})
%      \text{ and}\\ 
%      u_{i,2n}&=&\tau_{[0,2n)}(v_{1,2n}^2\dots v_{i,2n}^2). 
% \end{eqnarray*}

\noindent These cylinders are non-empty since  
    \[
    \tau_{[0,n)}(v_{1,n}^2\dots v_{i-1,n}^2v_{i,n}v_{i+1,n}),\tau_{[0,n)}(v_{1,n}^2\dots v_{i-1,n}^2v_{i,n}^2)\in\mathcal{L}(X_{\BT}).
    \]

\begin{figure}[ht]
\begin{center}
\begin{tikzpicture}
\beginpgfgraphicnamed{situation-b}
\begin{scope}[very thick]

\draw[loosely dotted] (-9,0) -- (-8,0) ;
\draw (-8,0)--(0,0);
\draw[blue] (0,0)--(4,0);
\draw[loosely dotted, blue] (4,0)--(5,0);
\draw (0,.1)--(0,-.1) node[above=1mm]{$0$};
\draw (-3,.1)--(-3,-.1) node[above=1mm]{$\alpha_{i,1}$};
\draw[snake=brace, mirror snake] (-3,-0.1)-- node[below] {$\tau_{0}(v_{1,1}^2\dots v_{i-1,1}^2v_{i,1})$} (0,-0.1) ;
\draw (1.5,.1)--(1.5,-.1) ;
\draw[snake=brace, mirror snake,blue] (0,-0.1)-- node[below] {$\tau_{0}(v_{i,1})$} (1.5,-0.1) ;

\draw (-7,.1)--(-7,-.1) node[above=1mm]{$\alpha_{i,2}$};
\draw[snake=brace, mirror snake] (-7,-.8)-- node[below] {$\tau_{[0,2)}(v_{1,2}^2\dots v_{i-1,2}^2v_{i,2})$} (-3,-.8) ;
\draw (3.5,.1)--(3.5,-.1) ;
\draw[snake=brace, mirror snake] (-2.96,-0.8)-- node[below] {$\tau_{[0,2)}(v_{i+1,2})$} (3.5,-0.8) ;
\end{scope}
\endpgfgraphicnamed
\end{tikzpicture}
\caption{Construction of cylinders to get $x^i$}
\label{fig:finite-asym}
\end{center}
\end{figure}
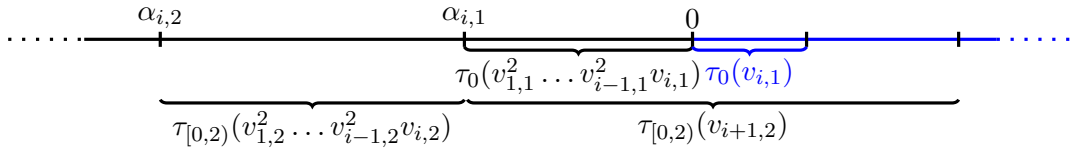
\begin{figure}[ht]
\begin{center}
\begin{tikzpicture}
\beginpgfgraphicnamed{situation-b}
\begin{scope}[very thick]

\draw[loosely dotted] (-9,0) -- (-8,0) ;
\draw (-8,0)--(0,0);
\draw[red] (0,0)--(4,0);
\draw[loosely dotted, red] (4,0)--(5,0);
\draw (0,.1)--(0,-.1) node[above=1mm]{$0$};
\draw (-3,.1)--(-3,-.1) node[above=1mm]{$\alpha_{i,1}$};
\draw[snake=brace, mirror snake] (-3,-0.1)-- node[below] {$\tau_{0}(v_{1,1}^2\dots v_{i-1,1}^2v_{i,1})$} (0,-0.1) ;
\draw (1.8,.1)--(1.8,-.1) ;
\draw[snake=brace, mirror snake, red] (0,-0.1)-- node[below] {$\tau_{0}(v_{i+1,1})$} (1.8,-0.1) ;

\draw (-7,.1)--(-7,-.1) node[above=1mm]{$\alpha_{i,2}$};
\draw[snake=brace, mirror snake] (-7,-.8)-- node[below] {$\tau_{[0,2)}(v_{1,2}^2\dots v_{i-1,2}^2v_{i,2})$} (-3,-.8) ;
\draw (3.5,.1)--(3.5,-.1) ;
\draw[snake=brace, mirror snake] (-2.96,-0.8)-- node[below] {$\tau_{[0,2)}(v_{i,2})$} (3.5,-0.8) ;
\end{scope}
\endpgfgraphicnamed
\end{tikzpicture}
\caption{Construction of cylinders to get $y^i$}
\label{fig:finite-asym1}
\end{center}
\end{figure}
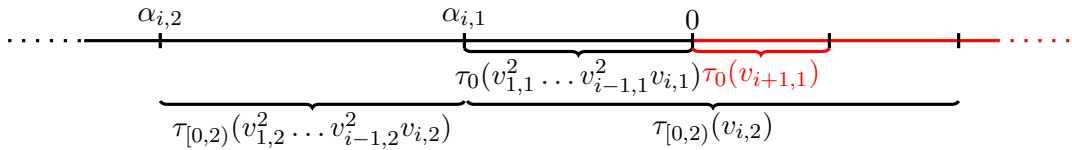    

%When $n$ is odd $\tau_{[0,n+1)}(v_{i+1,n+1})$ is a prefix of $w_{i,n}$ and when $n$ is even $\tau_{[0,n+1)}(v_{i,n+1})$ is a prefix of $w_{i,n}$. 
Since $w_{i,n}\prec w_{i,n+1}$ for all $n\ge 1$, it can be deduced that the cylinders are nested sequences whose diameters converge to zero. We let $\{x^i\}=\bigcap_{n\ge 1}C_{w_{i,n},\alpha_{i,n}}$ and $\{y^i\}=\bigcap_{n\ge 1}C_{u_{i,n},\alpha_{i,n}}$. Clearly, $x^i\ne y^i$ and $x^i\sim y^i$ as $x^i_{(-\infty,0)}=y^i_{(-\infty,0)}$ and $x^i_0\ne y^i_0$ (we refer to Figures~\ref{fig:finite-asym} and~\ref{fig:finite-asym1}). 

%\begin{proof}[Proof of Claim~\ref{Claim:1}]\renewcommand{\qedsymbol}{\ensuremath{\triangleright}}
Note that $\BT$ is a primitive morphism and hence $X_{\BT}$ is minimal. Also, since $(x^1,y^1)\in\mathcal{AP}(X_{\BT})$, $X_{\BT}$ is not periodic.
Hence, for each $1\le i\le k$, $x^i$ and $y^i$ are not in the same orbit (because otherwise, both $x^i$ and $y^i$ are periodic). 

Now we see that for $1\le i\ne j\le k,\ x^i$ and $x^j$ are not in the same orbit. 
Let $x=x^i,\ y=x^j$.  
Without loss of generality, assume that $S^t(x)=y$ for some $t>0$. We can choose $n>1$ large enough and even so that $\alpha_{i,n}+|w_{i,n}|>t$.  Let $(k_n,x^{(n)}),(\ell_n,y^{(n)})$
denote the centered $\tau_{[0,n)}$- representation of $x$ and $y$ in $X_{\BT}^{(n)}$, respectively. Clearly, $x^{(n)}_0=v_{i+1,n}\ne v_{j+1,n}=y^{(n)}_0$, $0\le k_n<|\tau_{[0,n)}(v_{i+1,n})|$ and $x=S^{k_n}(\tau_{[0,n)}(x^{(n)}))$. This implies that $y=S^t(x)=S^{k_n+t}(\tau_{[0,n)}(x^{(n)}))$.
It is easy to see that $k_n=-\alpha_{i,n-1}$ and $|\tau_{[0,n)}(v_{i+1,n})|=|w_{i,n}|-\alpha_{i,n-1}+\alpha_{i,n}.$ Hence, $0\le k_n+t=|\tau_{[0,n)}(v_{i+1,n})|-|w_{i,n}|-\alpha_{i,n}+t<|\tau_{[0,n)}(v_{i+1,n})|$. This means that $(k_n+t,x^{(n)})$ is a centered $\tau_{[0,n)}$- representation of $y$ in $X_{\BT}^{(n)}$ which contradicts its uniqueness. 
% Let $x=x^i,\ y=y^i,\ \alpha_n=\alpha_{i,n},\ p_n=|\tau_{[0,n)}(v_{i,n})|$ and $q_n=|\tau_{[0,n)}(v_{i+1,n})|$ .  
% Without loss of generality, assume that $S^t(x)=y$ for some $t>0$, that is, $y_{j}=x_{j+t}$ for all $j\in\Z$. We can choose $n$ large enough so that $\alpha_{n+1}+|w_{i,n+1}|>t$. We can also choose $n$ to be an odd number. 
% Note that $x_{[\alpha_{n},\alpha_{n}+b_{n+1}-1]}=\tau_{[0,n+1)}(v_{i+1,n+1})$ and $y_{[\alpha_{n},\alpha_{n}+a_{n+1}-1]}=\tau_{[0,n+1)}(v_{i,n+1})$.
% By the recognizability of $X_{\BT}$, $x_{[\alpha_{n}+b_{n+1}-c_{n},\alpha_{n}+b_{n+1}-d_{n}-1]}=\tau_{[0,n)}(v_{m_{n},n})\tau_{[0,n)}(v_{1,n})$ and $x_{[\alpha_{n},\alpha_{n}+q_{n+1}-1]}$ does not contain $\tau_{[0,n)}(v_{m_{n},n})\tau_{[0,n)}(v_{1,n})$ as a subword. 
%By the recognizability of $X_{\BT}$, $x_{[\alpha_{n-1}+q_n,\infty)}$ has a prefix $\tau_{[0,n-1)}(v_{1,n-1})$,  $x_{[\alpha_{n-1},\alpha_{n-1}+q_n-1]}$ has a suffix $\tau_{[0,n-1)}(v_{m_{n-1},n-1})$ and $x_{[\alpha_{n-1},\alpha_{n-1}+q_n-1]}$ does not contain $\tau_{[0,n-1)}(v_{m_{n-1},n-1})\tau_{[0,n-1)}(v_{1,n-1})$ as a subword. 
%There exists a subsequence $(n_m)_{m\ge 1}$ of positive integers such that $m_{i,n_m}+|w_{i,n_m}|>|\tau_{[0,n_m-1)}(v_{k+1,n_m-1})|$. Hence, 
% We can choose $n$ large enough so that $m_{i,n}+|w_{i,n}|>t$. 
% Let $r\le 0$ be the cutting point of $x$ by $\tau_{[0,n)}$. 
% If $x^{(n)},y^{(n)}$
% denote the $\tau_{[0,n)}$- representation of $x$ and $y$ in $X_{\BT}^{(n)}$, 
% then we have $x^{(n)}_0=y^{(n)}_0$.
% But $\{x^{(n)}_0,y^{(n)}_0\}=\{v_{i,n},v_{i+1,n}\}$ which is a contradiction. 
A similar argument also shows that for $i\ne j$, $x^i$ and $y^j$ are not in the same orbit. Hence, the pairs $(x^1,y^1),\dots,(x^k,y^k)$ contribute to asymptotic components. 

% Choose $n$ such that $\left<\tau_{[0,n)}\right>>t$. 
% \textcolor{red}{Complete the proof}
%\end{proof}

\begin{claim}\label{Claim:1}
    Let $(x,y)$ be a representative of an asymptotic component of $X_{\BT}$. Then $x\sim_{\text{Orb}}x^i$ for some $1\le i\le k$. 
\end{claim}
\begin{proof}[Proof of Claim~\ref{Claim:1}]\renewcommand{\qedsymbol}{\ensuremath{\triangleright}}
%On the contrary, assume that there exist $k+1$ many distinct non-trivial asymptotic pairs given by $(x^0,y^0),\dots, (x^k,y^k)$. 
Without loss of generality, assume that $x_{(-\infty,0)}=y_{(-\infty,0)}$ and $x_0\ne y_0$. Using the recognizability of $X_{\BT}$, choose $x^{(n)}, y^{(n)}$, $0\le r(n)< |\tau_{[0,n)}(x^{(n)}_0)|$ and $0\le t(n)< |\tau_{[0,n)}(y^{(n)}_0)|$ such that $x=S^{r(n)}(\tau_{[0,n)}(x^{(n)}))$ and $y=S^{t(n)}(\tau_{[0,n)}(y^{(n)}))$. Clearly, $x^{(n)}_{-1}=v_{i,n}$ for some $1\le i\le k$. Moreover by Lemma~\ref{Lemma:splittingpoint}, $r(n)=t(n)$ and $x_{[-r(n),0)}=y_{[-r(n),0)}$ is the common prefix of $\tau_{[0,n)}(v_{i,n})$ and $\tau_{[0,n)}(v_{i+1,n})$. This implies that $x_{[-s(n),0)}=x^i_{[-s(n),0)}$ where $s(n)=r(n)+|\tau_{[0,n)}(v_{i,n})|$ (Figure~\ref{fig:factorization}).  Since $s(n)\to\infty$ as $n\to \infty$, we get $x$ and $x^i$ are asymptotic.   

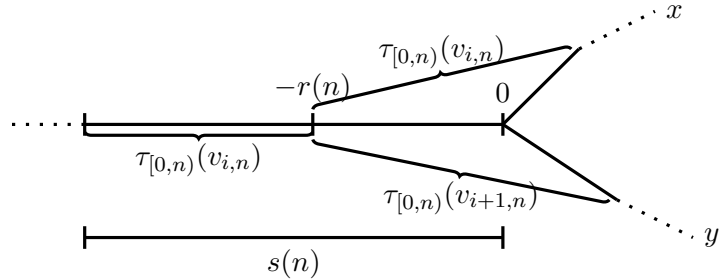
\begin{figure}[ht]
\begin{center}
\begin{tikzpicture}
\beginpgfgraphicnamed{situation-b}
\begin{scope}[very thick]
\draw[snake=brace] (3,0.2)-- node[above] {$\tau_{[0,n)}(v_{i,n})$}(6.4,1) ;
\draw[loosely dotted] (0,0) -- (-1,0);
\draw  (3,0) --  (5.5,0) ;
%\draw[snake=brace] (3,0.1)-- node[above] {$w$}(5.5,0.1) ;
\draw  (5.5,0)  -- (6.5,1);
\draw (0,0) --  (3,0);
\draw[snake=brace, mirror snake] (0,-0.1)-- node[below] {$\tau_{[0,n)}(v_{i,n})$}(3,-0.1) ;
\draw (3,-0.15) -- (3,0.15)  node[anchor=south] {$-r(n)$} ;
\draw (5.5,0) -- (7,-1);
\draw[loosely dotted] (7,-1)-- (8,-1.5)node[anchor=west] {$y$};
\draw[loosely dotted] (6.5,1)-- (7.5,1.5)
node[anchor=west] {$x$};
\draw (0,-0.15) -- (0,0.15);
\draw[snake=brace, mirror snake] (3,-0.2) -- node[below] {$\tau_{[0,n)}(v_{i+1,n})$} (6.9,-1) ;
\draw (5.5,-0.15) -- (5.5,0.15) node[anchor=south] {$0$} ;
\draw (0,-1.5)-- node[below] {$s(n)$} (5.5,-1.5);
\draw (0,-1.35)--(0,-1.65);
\draw (5.5,-1.35)--(5.5,-1.65);
\end{scope}
\endpgfgraphicnamed
\end{tikzpicture}
\caption{Factorization of $x$ and $y$ at level $n$}
\label{fig:factorization}
\end{center}
\end{figure}

\end{proof}

By Lemma~\ref{Lemma:twocomp}, no two pairs $(x^i,y^i)$ and $(x^j,y^j)$, $i\ne j$
can be in the same asymptotic component. Hence, these are at least $k$ distinct asymptotic components. By Claim~\ref{Claim:1}, the subshift $X_{\BT}$ has exactly $k$ many asymptotic components. 
\end{proof}
%\begin{remark}\label{rem:2}
 %   If $v_{i,n}\in \tilde{R}_{X_{\BT}^{(n)}}(1)$ for $1\le i\le k$, then there exist $a\ne b$ such that for all $t$, there exists a word $w$ of length $t$ where $wv_{i,n}a,wv_{i,n}b\in\mathcal{L}(X_{\BT}^{(n)})$. By the definition of the morphisms, if $t$ is large enough, then $w$ has a suffix given by $v_{m_n,n}v^2_{1,n}\dots v_{i-1,n}^2$ and $\{a,b\}=\{v_{i,n},v_{i+1,n}\}.$
%\end{remark}
%\begin{remark}\label{rem:2}
 % For $n\ge 1$, let $x\ne y\in X_{\BT}^{(n)}$ be asymptotic. Without loss of generality assume that $x_{(-\infty,-1]}=y_{(-\infty,-1]}$ and $x_0\ne y_0$. Then $x_{-1}=y_{-1}=v_{i,n}$ for some $1\le i\le k$. Then $x_{[-2i,-1]}=y_{[-2i,-1]}=v_{m_n,n}v_{1,n}^2\dots v_{i-1,n}^2 v_{i,n}$ and $\{x_0,y_0\}=\{v_{i,n},v_{i+1,n}\}$.
%\end{remark}

\begin{corollary}
     Let $\BT=(\tau_n)_{n\ge 0}$ be a directive sequence of morphisms that satisfies Property ($P_k$) and let $X_{\BT}$ be the associated $\mathcal{S}$-adic subshift. Then $(X_{\BT},S)$ has zero entropy. 
\end{corollary}
\begin{proof}
    By Proposition~\ref{prop:kasymToep}, the $\mathcal{S}$-adic subshift $X_{\BT}$ has only finitely many asymptotic components. Hence, the result is followed by Theorem~\ref{thm:main3}.
\end{proof}

In order to prove our main theorem, it is enough to construct an $\mathcal{S}$-adic subshift associated with a directive sequence of morphisms that satisfies Property ($P_k$) within a given strong orbit equivalence class. The following lemma allows us to construct such a subshift~\cite[Section 4.1]{CD20}.
\begin{lemma}\label{lemma:A>V}
      Let $B=(V,E)$ be a simple Bratteli diagram. There exists a simple Bratteli diagram $B'=(V',E')$ with a sequence of incidence matrices $(A_n')_{n\in\N}$ such that for all $n\geq 1$, $|V'_n|>n$,
      all the entries of $A'_n$ is at least $|V_{n+1}'|$ and $(X_B,T_B)$ and $(X_{B'},T_{B'})$ are strong orbit equivalent. 
\end{lemma}

\begin{proposition}\label{prop:conjugacy2} 
Let $B=(V,E,\le)$ be a simple ordered Bratteli diagram. Then, there exists a simple ordered Bratteli diagram $B'=(V',E',\le')$ such that the Bratteli-Vershik systems $(X_B,V_B)$ and $(X_{B'},V_{B'})$ are strong orbit equivalent and the directive sequence of morphisms read on $B'$,
 given by $\BT=(\tau_n)_{n\ge 0}$, satisfies Property ($P_k$).
   Moreover,
   the $\mathcal{S}$-adic subshift $(X_{\BT},S)$ is conjugate to $(X_{B'},V_{B'})$.  
\end{proposition}
\begin{proof}
    By Lemma~\ref{lemma:A>V}, and telescoping if necessary, we assume that for all $n\ge 1$, $|V'_n|>k$, and $(A_n')_{i,j}\ge |V'_{n+1}|$.
    We let $V'_n=\{v_{1,n},\dots,v_{m_n,n}\}$. We can easily define an ordering $\le'$ on $(V',E')$ such that the directive sequence of morphism $\BT=(\tau_n)_{n\ge 0}$ read on $B'=(V',E',\le')$ satisfies Property ($P_k$).

    To show that $(X_{\BT},S)$ is conjugate to $(X_{B'},V_{B'})$,
    by Proposition~\ref{prop:conjugacy}, it is enough to prove that $\tau_n$ is proper for $n\ge 1$ and that $\tau_n$ extends by concatenation to an injective map from $X_{\BT}^{(n+1)}$ to $X_{\BT}^{(n)}$. Clearly, $\tau_n$ is proper. 
    
 We fix two distinct sequences $x=(x_i)_{i\in\Z},x'=(x'_i)_{i\in\Z}\in X_{\BT}^{(n+1)}$
   such that $\tau_n(x)=\tau_n(x')$. 
    % Since $v_{m_n,n}^{\ell_{i,n}}$ is a suffix of $\tau_n(v_{i,n+1})$, the symbol $v_{m_n,n}$ does not appear anywhere else in the middle. 
     The cutting points of $\tau_n(x)$ are wherever $v_{m_n,n}$ is followed by $v_{1,n}$. Since the word $v_{m_n,n}v_{1,n}$ does not appear anywhere else except at the cutting points, the cutting points of $\tau_n(x)$ and $\tau_n(x')$ are the same. Hence $\tau_n(x_i)=\tau_n(x'_i)$ for all $i\in\Z$. Since $\tau_n$ is injective on symbols, we have $x_i=x_i'$ for all $i\in\Z$.  
\end{proof}

%We now have all the necessary components to prove  Theorem~\ref{thm:main1}. 
\begin{theorem}\label{thm:kmany}
    For any given $k\ge 1$ and for any given minimal Cantor system $(X,T)$, there exists a subshift which has exactly $k$ many asymptotic components and that is strong orbit equivalent to $(X,T)$. 
\end{theorem} 
\begin{proof}
    Let $B=(V,E,\le)$ be the simple ordered Bratteli diagram with respect to which the associated Bratteli-Vershik system is conjugate to $(X,T)$ (Theorem~\ref{thm:minimalCantor}).

 By Proposition~\ref{prop:conjugacy2}, construct an ordered Bratteli diagram $B'=(V,E',\le')$ such that $(X_B,V_B)$ is strong orbit equivalent to $(X_{B'},V_{B'})$ and the directive sequence of morphisms read on $B'$, given by $\BT$, satisfies Property ($P_k$).  By Lemma~\ref{lemma:kasym}, the subshift $(X_{\BT},S)$ has exactly $k$ many asymptotic components. Again, using Proposition~\ref{prop:conjugacy2}, $(X_{B'},V_{B'})$ is conjugate to $(X_{\BT},S)$ and hence, $(X,T)$ and $(X_{\BT},S)$ are strong orbit equivalent. 
\end{proof}
Following directly from Proposition~\ref{prop:aut} on the $\mathcal{S}$-adic subshift that satisfies Property $(P_1)$, we have Corollary~\ref{cor:aut}.

\subsection{Invariant measures and automorphism groups}
For a topological dynamical system $(X,T)$, a Borel probability measure $(X,\CB,\mu)$ is said to be \emph{invariant} if for any Borel subset $A\in\CB$, we have $\mu(T^{-1}(A))=\mu(A)$. It is well known that the set of invariant probability measures, denoted as $\CM(X,T)$, endowed with the $\text{weak}^*$ topology, is a non-empty \emph{Choquet simplex}. Recall that a compact, convex, metrizable subset $K$ of a locally convex real vector space is said to be a Choquet simplex if for every $v\in K$, there exists a unique probability measure $\mu$ on the set of extreme points of $K$, denoted by ${\rm ext}(K)$, such that $\int_{{\rm ext}(K)}xd\mu(x)=v.$
We now discuss the possible invariant measures of systems with trivial automorphism groups. 
\begin{corollary}\label{cor:invmeas}
    Let $K$ be any Choquet simplex. Then there exists a subshift $(Y,S)$ with a trivial automorphism group such that its set of invariant measures, $\CM(Y,S)$ is affine homeomorphic to $K$.
\end{corollary}
\begin{proof}
    By~\cite[Theorem 5]{Downa91}, there exists a subshift $(X,S)$ such that $\CM(X,S)$ is affine homeomorphic to $K$. Let $(X',S)$ be the subshift that is strong orbit equivalent to $(X,S)$ with $\text{Aut}(X',S)$ generated by $S$. Since the strong orbit equivalence preserves the set of invariant measures, up to the affine homeomorphism, we can conclude that $\CM(X',S)$ is affine homeomorphic to $K$. 
\end{proof}

For a topological dynamical system $(X,T)$, a Borel probability measure $(X,\CB,\mu)$ is said to be a \emph{characteristic measure} if for any Borel subset $A\in\CB$, we have $\mu(f^{-1}(A))=\mu(A)$ for all $f\in\text{Aut}(X,T)$. The existence of a characteristic measure, in general, is still unknown. In the case of subshifts, by Krylov-Bogolioubov theorem, the characteristic measures exist if the automorphism group is amenable. Also, characteristic measures exist for subshifts with zero entropy~\cite{FT22}. In any case, we prove an analogue of~\cite[Theorem 5]{Downa91} in the case of characteristic measures. 

\begin{corollary}
     Let $K$ be any Choquet simplex. Then there exists a subshift $(Y,S)$ such that its set of characteristic measures is affine homeomorphic to $K$.
\end{corollary}
\begin{proof}
    By Corollary~\ref{cor:invmeas}, there exists a subshift $(Y,S)$ with a trivial automorphism group such that its set of invariant measures, $\CM(Y,S)$ is affine homeomorphic to $K$. The result holds for this subshift, since the set of all characteristic measures is given by $\CM(Y,S)$. 
\end{proof}

\subsection{Complexity functions and automorphism groups}
In this section, within a strong orbit equivalence class of a minimal Cantor system, we construct subshifts that have arbitrarily small superlinear complexity and trivial automorphism group. We also show that \cref{thm:main3} is optimal. Specifically, we construct subshifts with arbitrarily large subexponential complexity that have a single asymptotic component (and consequently, a trivial automorphism group).
\begin{corollary}\label{thm:complexity}
    Let $(X,T)$ be a minimal Cantor system, and $(p_n)_{n\ge 1}$ be a sequence of positive real numbers such that $\lim_{n\to\infty}\frac{n}{p_n}=0$. Then there exists a subshift $(Y,S)$ satisfying the following;
    \begin{enumerate}
        \item $(X,T)$ and $(Y,S)$ are strong orbit equivalent,
         \item The complexity of $(Y,S)$ satisfies $\lim_{n\to\infty} \frac{p_Y(n)}{p_n}=0,$ and
        \item $\text{Aut}(Y,S)$ is generated by $S$.
    \end{enumerate}
\end{corollary}
\begin{proof}
    Note that the directive sequence of morphisms that is constructed to satisfy (1) and (2) in~\cite[Theorem 1.2]{CD20} satisfies Property $(P_1)$. 
\end{proof}

In the following result, we show the existence of subshifts with sub-exponential complexity having trivial automorphism group. 

\begin{theorem} \label{thm:subexponential}
    Let $(g_n)_{n\ge 1}$ be a sequence of positive real numbers such that $\lim_{n\to\infty}\frac{\ln(g_n)}{n}=0$. Then there exists a subshift $(X,S)$ having exactly one asymptotic component (and thus having the trivial automorphism group) such that $\liminf
    _{n\to\infty}\frac{g_n}{p_X(n)}=0.$
\end{theorem}
\begin{proof}
    We construct a Bratteli-Vershik system as follows. Let $|V_n|=n+2$ for all $n\ge 0$. For $u\in V_{n+1}$, let $L_n(u)$ denote the total number of edges from all the vertices in $V_n$ to $u$. In our construction, each $L_n(u)$ is the same\footnote{Thus, our example is a Toeplitz subshift.}, and we denote it as $L_n$. Let $L_0=3$. Now we define the rest of the levels inductively. Let $L_0,\dots,L_{n-1}$ be given. We choose $\alpha_{n}\in\N$ such that $(n+1)^{\alpha_{n}-1}\ge g_{\alpha_nL_{n-1}\dots L_0}$. It is possible to find such an $\alpha_n$ since $(g_n)_{n\ge 0}$ is sub-exponential. 
    
    We give an ordering on each level so that the morphism read on it will satisfy Property $(P_1)$. In particular, if $V_n=\{v_{1,n},\dots,v_{n+2,n}\},$ we have $\tau_n(v_{1,n+1})=v_{1,n}v_{2,n}wv_{n+2,n}$ where $w$ is a word of length $L_n-3$ that does not contain $v_{n+2,n}v_{1,n}$ as a subword. Since there are no other conditions on $w$, by choosing a large enough $L_n$, we can ensure that $w$ contains all possible words of length $\alpha_n$ with symbols from $V_n\setminus\{v_{1,n}\}$, as subwords. There are $(n+1)^{\alpha_n}$ many words of length $\alpha_n$ with symbols from $V_n\setminus\{v_{1,n}\}$. For each distinct word $u$ of length $\alpha_n$,  since the morphism is injective on symbols, $\tau_{[0,n)}(u)$ is distinct and is of length $\alpha_nL_{n-1}\dots L_1$. For $i>1$, we can define $\tau_n(v_{i,n+1})=v_{1,n}^iv_{2,n}w'v_{n+2,n}$ where $|w'|=L_n-3$ and $w'$ does not contain $v_{n+2,n}v_{1,n}$ as a subword. 
    
    Let $X$ be the $\mathcal{S}$-adic subshift generated by this directive sequence of morphisms. Clearly, 
    \[
    p_X(\alpha_nL_{n-1}\dots L_0)\ge (n+1)^{\alpha_n}\ge (n+1)g_{\alpha_nL_{n-1}\dots L_0}.
    \]
    Let $\tilde{L}(n)=\alpha_nL_{n-1}\dots L_0$. Then we have $\frac{g_{\tilde{L}_n}}{p_X(\tilde{L}_n)}\le \frac{1}{n+1}$. Hence, $\liminf
    _{n\to\infty}\frac{g_n}{p_X(n)}=0.$   
\end{proof}

\section{Subshifts with countably infinite asymptotic components}\label{Sec:countable}
Given a minimal Cantor system $(X,T)$, we construct a subshift that is strong orbit equivalent to $(X,T)$ and that has countably infinite asymptotic components. 

%Let $(V,E)$ be a simple Bratteli diagram with the sequence of incidence matrices $(A_n)_{n\ge 0}$. By telescoping, we assume that $(A_n)_{i,j}>0$ and $(A_n)_{i,1}>|V_{n+1}|$ for all $n\ge 1, 1\le i\le m_{n+1}$ and $1\le j\le m_n$ where $|V_k|=m_k$. Additionally, we also assume that $|V_n|>n$ for all $n\ge 1$ (for this, we need the topological rank to be not finite).

\subsection{A directive sequence of morphisms} 
Given a simple Bratteli diagram, we define the following property of a directive sequence of morphisms. The $\mathcal{S}$-adic subshift generated by this sequence will be shown to have countably many asymptotic components.

\begin{definition}\label{def:countable-asym}
Let $\BT=(\tau_n:V^*_{n+1}\to V^*_n)_{n\ge 0}$ be a directive sequence of morphisms. We say that $\BT$    
    satisfies \emph{Property ($P_\infty$)} if the following conditions are true. For $n\ge 1$, we have
    \begin{enumerate}
    \item $\tau_n$ is primitive,
        \item for $1\le i\le n+1$, $v_{1,n}^2v_{2,n}^2\dots v_{i-1,n}^2v_{i,n}\dots v_{n,n}v_{n+1,n}$ is a prefix of $\tau_n(v_{i,n+1})$ , 
        \item for $ i> n+1$, $v_{1,n}^iv_{2,n}$ is a prefix of $\tau_n(v_{i,n+1})$, and 
        \item for $ 1\le i\le m_{n+1}$, $v_{m_n,n}$ is a suffix
        of $\tau_n(v_{i,n+1})$, and it does not have $v_{m_n,n}v_{1,n}$ as a subword,  \end{enumerate}
where $V_n=\{v_{1,n},\dots,v_{m_n,n}\}$. For $n=0,$ $\tau_0$ is a hat morphism. 
\end{definition}
Similar to the last section, whenever we talk about Property ($P_\infty$), we specify the ordering on $V_n$ as $V_n=\{v_{1,n},\dots,v_{m_n,n}\}$, where $v_{1,n},\dots,v_{n+1,n}$ being the first $n+1$ distinguished letters. 
Also, note that $\tau_n$ is injective on symbols. 
%We have the following lemma, whose proof is similar to that of in the last section. 

 Consider a minimal Cantor system $(X,T)$. Let $(V,E)$ be the associated Bratteli diagram and $A_n$ be the associated incidence matrix at level $n$. Telescoping if necessary and using Lemma~\ref{lemma:A>V}, assume that $|V_n|>n$ and $(A_n)_{i,j}\ge |V_{n+1}|$ for all $n\ge 1$. 
 We can define a new ordering $\le$ on $(V,E)$, such that the directive sequences of morphisms $\BT=(\tau_n)_{n\ge 0}$ read on $(V,E,\le)$ satisfies Property $(P_\infty)$. Using Proposition~\ref{prop:soe} and Proposition~\ref{prop:conjugacy}, the $\mathcal{S}$-adic subshift $X_{\BT}$ is strong orbit equivalent to $(X,T)$. We state this as a lemma below for convenience.
 
\begin{lemma}\label{lemma:countable-asym}
    Let $(X,T)$ be a minimal 
    Cantor system. Then there exists a ordered Bratteli diagram $B'=(V',E',\le')$ such that $(X_{B'},V_{B'})$ is strong orbit equivalent to $(X,T)$ and
    the directive sequence of morphisms, $\BT$, read on $B'$ satisfies Property ($P_\infty$). Moreover, $(X_{\BT},S)$ and $(X_{B'},V_{B'})$ are conjugate. 
    
%     $\BT$ be the directive sequence of  morphisms that satisfy the properties as in Definition~\ref{def:finite-asym}. Let $B=(V,E,\ge')$ be the ordered Bratteli diagram where 
% the directive sequence of morphisms read on $B$ is
%  given by $\BT$, $X_{\BT}^{(n)}$ be the $\mathcal{S}$-adic subshift at level $n$ and $X_{\BT}=X_{\BT}^{(0)}$. Also, let $V_n=\{v_{1,n},\dots,v_{m_n,n}\}.$
%   Then the following are true.
%     \begin{enumerate}
%         \item   The $\mathcal{S}$-adic subshift $(X_{\BT},S)$ is conjugate to the Bratteli-Vershik system $(X_B,V_B)$.  
%         \item For $n\ge 1$, the words $\tau_{[0,n)}(v_{i,n})$ and $\tau_{[0,n)}(v_{i+1,n})$ are not prefix dependent where $i=1,\dots,n$.
%         \item There are at most $k$ many extendable right special word of length 1 in $X_{\BT}^{(n)}$. In particular,   
%     $\tilde{R}_{X_{\BT}^{(n)}}(1)\subseteq\{v_{1,n},\dots,v_{n,n}\}$.
%     \end{enumerate}
\end{lemma}

\begin{remark}\label{Rem:prefixcount}
 Note that  the conclusions of Lemmas~\ref{lemma:prefix_indep},~\ref{lemma:suffix},~\ref{Lemma:splittingpoint} hold true when $\BT$ satisfies Property ($P_\infty$). That is, $v_{i,n}$ is a signal for sequences in $X_{\BT}^{(n)}$ if $1\le i\le n$ and at the bifurcation point, it is always followed by $v_{i,n}$ or $v_{i+1,n}$.  
 However, Lemma~\ref{Lemma:prefix} holds true only for the signals $v_{i,n}$ where $1<i\le n$ That is, for $i>1$, when $v_{i,n}$ is a signal, at the bifurcation point, it is only preceded by $v_{m_n,n}v_{1,n}^2\dots v_{i-1,n}^2$. A modification to Lemma~\ref{Lemma:prefix} has to be stated when $v_{1,n}$ is a signal. In this case, $v_{1,n}$ is preceded either by $v_{m_n,n}$ or by $v_{m_n,n}v_{1,n}$. The second case is possible because unlike in the case of Property ($P_k$), the letter $v_{n+1,n+1}$ can be a signal for sequences in $X_{\BT}^{(n+1)}$ and it is followed by $v_{n+1,n+1}$ or $v_{n+2,n+1}$. Note that,
 $v_{1,n}^2v_{2,n}$ is a prefix of $\tau_n(v_{n+1,n+1})$ and $v_{1,n}^{3}$ is a prefix of $\tau_n(n+2,n+1)$. In any case, Lemma~\ref{Lemma:prefix} is only used to prove Lemma~\ref{Lemma:twocomp} and we shall see later that we still get the required result. 
\end{remark}

\begin{lemma}\label{Lemma:splitinf}
    Let $\BT=(\tau_n:V_{n+1}^*\to V_n^*)_{n\ge 0}$ satisfy Property $(P_\infty)$ where $V_n=\{v_{1,n},\dots,v_{m_n,n}\}$. Let $x,y\in X_{\BT}$ be such that $x_{(-\infty,0)}=y_{(-\infty,0)}$ and $x_0\ne y_0$ and $x^{(n)},y^{(n)}$ be their $\tau_{[0,n)}$-factorization, respectively for $n\ge 1$. Then there exist $N\ge 1$ and $1\le i\le N$ such that $J(x^{(n)},y^{(n)})=(v_{i,n},0)$ for all $n\ge N$. 
\end{lemma}
\begin{proof}

Let $J(x^{(N)},y^{(N)})=(v_{i,N},0)$ for some $1<i\le N$ and $N\ge 1$. Then $x^{(N)}_{[-2i,-1]}=v_{m_N,N}v_{1,N}^2\dots v_{i-1,N}^2v_{i,N}=y^{(N)}_{[-2i,-1]}$ and $\{x^{(N)}_0,y^{(N)}_0\}=\{v_{i,N},v_{i+1,N}\}$. This implies that $$\{x^{(N+1)}_0,y^{(N+1)}_0\}=\{v_{i,N+1},v_{i+1,N+1}\}$$ and hence, $J(x^{(N+1)},y^{(N+1)})=(v_{i,N+1},0). $ Inductively, $J(x^{(n)},y^{(n)})=(v_{i,n},0),$ for all $n\ge N$.

When $J(x^{(N)},y^{(N)})=(v_{1,N},0)$ for some $N\ge 1$, we have $\{x^{(N)}_0,y^{(N)}_0\}=\{v_{1,N},v_{2,N}\}$ and two possible scenarios: either $x^{(N)}_{[-2,-1]}=v_{m_N,N}v_{1,N}=y^{(N)}_{[-2,-1]}$ or $x^{(N)}_{[-3,-1]}=v_{m_N,N}v^2_{1,N}=y^{(N)}_{[-3,-1]}$. We first investigate the second scenario. In this case, by Remark~\ref{Rem:prefixcount}, we have $J(x^{(N+1)},y^{(N+1)})=(v_{N+1,N+1},0)$ and hence $J(x^{(n)},y^{(n)})=(v_{N+1,n},0)$ for all $n\ge N+1$ by the previous argument. Now we focus on the first scenario. In this case, we have, $J(x^{(N+1)},y^{(N+1)})=(v_{1,N+1},0)$. Now we can further divide this into two cases: either $x^{(n)}_{[-2,-1]}=v_{m_n,n}v_{1,n}=y^{(n)}_{[-2,-1]}$ for all $n>N$, or $x^{(n_0)}_{[-3,-1]}=v_{m_{n_0},n_0}v^2_{1,n_0}=y^{(n_0)}_{[-3,-1]}$ for some $n_0>N$. In the first case we have $J(x^{(n)},y^{(n)})=(v_{1,n},0)$ for all $n\ge N$ and in the second case we have $J(x^{(n)},y^{(n)})=(v_{n_0+1,n},0)$ for all $n>n_0$. 
\end{proof}

\subsection{Main results}

\begin{theorem}\label{thm:count}
    For any given minimal Cantor system $(X,T)$, there exists a subshift which has countably infinite asymptotic components and that is strong orbit equivalent to $(X,T)$. 
\end{theorem} 
\begin{proof}
By Lemma~\ref{lemma:countable-asym}, it is enough to show that if $\BT$ satisfies Property $(P_\infty)$, then $X_{\BT}$ has countably infinite asymptotic components.   
We construct countably infinite asymptotic pairs for $X_{\BT}$, say $\{(x^n,y^n)\mid n\ge 1\}$. We then prove that these pairs contribute to asymptotic components, no two asymptotic components are the same and that these are the only possible ones. 

For $i\in\N$, we define two decreasing nested sequences $(C_{w_{i,n},m_{i,n}})_{n\ge i}$ and $(C_{u_{i,n},\alpha_{i,n}})_{n\ge i}$ of cylinders on $X_{\BT}$ where $w_{i,n},u_{i,n}\in\mathcal{L}(X_{\BT})$ and $\alpha_{i,n}<0$ are defined inductively. 
Let $\alpha_{i,i}=-|\tau_{[0,i)}(v_{1,i}^2\dots v_{i-1,i}^2v_{i,i})|$ and for $n>i,$
\[\alpha_{i,n}=-|\tau_{[0,n)}(v_{1,n}^2\dots v_{i-1,n}^2v_{i,n})|+\alpha_{i,n-1}.\] 
For $i\ge 1,n\ge 0$, define 
\begin{equation*}
    w_{i,i+n}=\begin{cases}
        \tau_{[0,i+n)}(v_{1,i+n}^2\dots v_{i-1,i+n}^2v_{i,i+n}^2), &\text{if } n\text{ is odd}\\
        \tau_{[0,i+n)}(v_{1,i+n}^2\dots v_{i-1,i+n}^2v_{i,i+n}v_{i+1,i+n}), &\text{if } n\text{ is even},
    \end{cases}
\end{equation*}
\begin{equation*}
    u_{i,i+n}=\begin{cases}
        \tau_{[0,i+n)}(v_{1,i+n}^2\dots v_{i-1,i+n}^2v_{i,i+n}v_{i+1,i+n}), &\text{if } n\text{ is odd}\\
        \tau_{[0,i+n)}(v_{1,i+n}^2\dots v_{i-1,i+n}^2v_{i,i+n}^2), &\text{if } n\text{ is even}.
    \end{cases}
\end{equation*}

% \begin{eqnarray*}
%     w_{n,n+2i}&=&\tau_{[0,n+2i)}(v_{1,n+2i}^2\dots v_{n,n+2i}v_{n+1,n+2i}),\\
%     w_{n,n+2i-1}&=&\tau_{[0,n+2i-1)}(v_{1,n+2i-1}^2\dots v_{n,n+2i-1}^2),
%     \\ 
%      u_{n,n+2i}&=&\tau_{[0,n+2i)}(v_{1,n+2i}^2\dots v_{n,n+2i}^2) \text{ and}\\
%     u_{n,n+2i-1}&=&\tau_{[0,n+2i-1)}(v_{1,n+2i-1}^2\dots v_{n,n+2i-1}v_{n+1,n+2i-1}). 
% \end{eqnarray*}

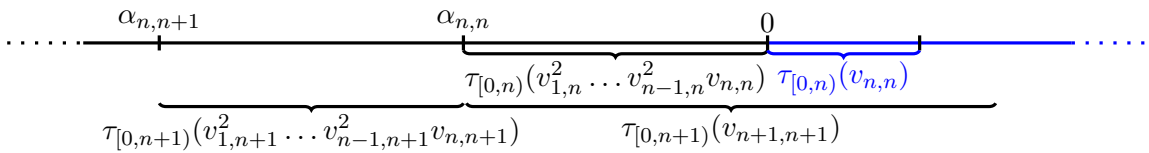
\begin{figure}[ht]
\begin{center}
\begin{tikzpicture}
\beginpgfgraphicnamed{situation-b}
\begin{scope}[very thick]

\draw[loosely dotted] (-9,0) -- (-8,0) ;
\draw (-8,0)--(1,0);
\draw[blue] (1,0)--(5,0);
\draw[loosely dotted, blue] (5,0)--(6,0);
\draw (1,.1)--(1,-.1) node[above=1mm]{$0$};
\draw (-3,.1)--(-3,-.1) node[above=1mm]{$\alpha_{n,n}$};
\draw[snake=brace, mirror snake] (-3,-0.1)-- node[below] {$\tau_{[0,n)}(v_{1,n}^2\dots v_{n-1,n}^2v_{n,n})$} (1,-0.1) ;
\draw (3,.1)--(3,-.1) ;
\draw[snake=brace, mirror snake,blue] (1,-0.1)-- node[below] {$\tau_{[0,n)}(v_{n,n})$} (3,-0.1) ;

\draw (-7,.1)--(-7,-.1) node[above=1mm]{$\alpha_{n,n+1}$};
\draw[snake=brace, mirror snake] (-7,-.8)-- node[below] {$\tau_{[0,n+1)}(v_{1,n+1}^2\dots v_{n-1,n+1}^2v_{n,n+1})$} (-3,-.8) ;
%\draw (3.5,.1)--(3.5,-.1) ;
\draw[snake=brace, mirror snake] (-2.96,-0.8)-- node[below] {$\tau_{[0,n+1)}(v_{n+1,n+1})$} (4,-0.8) ;
\end{scope}
\endpgfgraphicnamed
\end{tikzpicture}
\caption{Construction of cylinders to get $x^n$}
\label{fig:countable-asym}
\end{center}
\end{figure}

\begin{figure}[ht]
\begin{center}
\begin{tikzpicture}
\beginpgfgraphicnamed{situation-b}
\begin{scope}[very thick]

\draw[loosely dotted] (-9,0) -- (-8,0) ;
\draw (-8,0)--(1,0);
\draw[blue] (1,0)--(5,0);
\draw[loosely dotted, blue] (5,0)--(6,0);
\draw (1,.1)--(1,-.1) node[above=1mm]{$0$};
\draw (-3,.1)--(-3,-.1) node[above=1mm]{$\alpha_{n,n}$};
\draw[snake=brace, mirror snake] (-3,-0.1)-- node[below] {$\tau_{[0,n)}(v_{1,n}^2\dots v_{n-1,n}^2v_{n,n})$} (1,-0.1) ;
\draw (3,.1)--(3,-.1) ;
\draw[snake=brace, mirror snake,blue] (1,-0.1)-- node[below] {$\tau_{[0,n)}(v_{n+1,n})$} (3,-0.1) ;

\draw (-7,.1)--(-7,-.1) node[above=1mm]{$\alpha_{n,n+1}$};
\draw[snake=brace, mirror snake] (-7,-.8)-- node[below] {$\tau_{[0,n+1)}(v_{1,n+1}^2\dots v_{n-1,n+1}^2v_{n,n+1})$} (-3,-.8) ;
%\draw (3.5,.1)--(3.5,-.1) ;
\draw[snake=brace, mirror snake] (-2.96,-0.8)-- node[below] {$\tau_{[0,n+1)}(v_{n,n+1})$} (4,-0.8) ;
\end{scope}
\endpgfgraphicnamed
\end{tikzpicture}
\caption{Construction of cylinders to get $x^n$}
\label{fig:countable-asym1}
\end{center}
\end{figure}
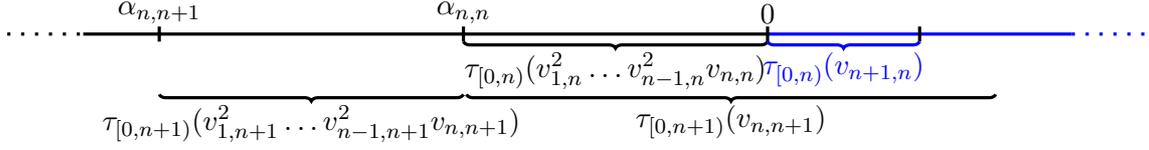

It is easy to check that $(C_{w_{i,n},\alpha_{i,n}})_{n\ge i}$ and $(C_{u_{i,n},\alpha_{i,n}})_{n\ge i}$ are decreasing nested sequences of cylinders and we define $x^i=\bigcap_{n\ge i}C_{w_{i,n},\alpha_{i,n}}$ and $y^i=\bigcap_{n\ge i}C_{u_{i,n},\alpha_{i,n}}$. We refer to Figures~\ref{fig:countable-asym}
and~\ref{fig:countable-asym1} to see the first two steps of the construction.

Note that $X_{\BT}$ is recognizable and minimal. 
Hence, a similar analysis as in the proof of Theorem~\ref{thm:kmany} using the $\tau_{[0,n)}$-factorizations for $n>N$ where $N$ is as in Lemma~\ref{Lemma:splitinf}, will confirm that $x^n$ and $y^n$ are not in the same orbit and for $n\ne m$, $x^n$ and $x^m$ are not orbit asymptotic. 

Now we are left to show that if $(x,y)$ is an representative of an asymptotic component for $x,y\in X_{\BT}$, then $x\sim_{Orb} x^i$ for some $i\in\N$. Without loss of generality, assume that $x_{(-\infty,0)}=y_{(-\infty,0)}$ and $x_0\ne y_0$. 
Take $n\ge N$ where $N$ is as in Lemma~\ref{Lemma:splitinf}. 
Let $x^{(n)},y^{(n)}\in X_{\BT}^{(n)}$ be the $\tau_{[0,n)}$-factorization of $x,y$, respectively. 
Clearly, $x^{(n)}_{-1}=v_{i,n}$ for some $1\le i\le n$. This implies that $x_{[-s(n),0)}=x^i_{[-s(n),0)}$ where $s(n)=r(n)+|\tau_{[0,n)}(v_{i,n})|$ and $r(n)$ is the length of the common prefix of $\tau_{[0,n)}(v_{i,n})$ and $\tau_{[0,n)}(v_{i+1,n})$.  Since $s(n)\to\infty$ as $n\to \infty$, we get $x$ and $x^i$ to be asymptotic. 
\end{proof}

\begin{corollary}
     Let $\BT=(\tau_n)_{n\ge 0}$ be a directive sequence of morphisms that satisfies Property ($P_\infty$) and let $X_{\BT}$ be the associated $\mathcal{S}$-adic subshift. Then $(X_{\BT},S)$ has zero entropy. 
\end{corollary}
\begin{proof}
    By Proposition~\ref{prop:kasymToep}, the $\mathcal{S}$-adic subshift $X_{\BT}$ has only countably many asymptotic components. Hence, the result is followed by Theorem~\ref{thm:main3}.
\end{proof}
\section{Asymptotic components within a strong orbit equivalence class of Toeplitz shifts}\label{Sec:Toeplitz}
\noindent 
In this section, we explore how our previous results can be realized within the context of Toeplitz subshifts. That is, we aim to find examples of Toeplitz subshifts with a prescribed number of asymptotic components within a strong orbit equivalence class. We show this is possible, under the obvious restriction that the strong orbit equivalence class contains at least one Toeplitz subshift. We give similar proofs to the general case, but the methods are slightly different since they are tailored to handle the Toeplitz case. 

A simple Bratteli diagram $(V,E)$ is said to have \emph{equal path number property} if for all $n\ge 0$ and $v,v'\in V_n,$ we have $\#(r^{-1}(v))=\#(r^{-1}(v'))$ where $\#(.)$ denotes the cardinality of a set. This is same as saying that the incidence matrices $(A_n)_{n\ge 0}$ have equal row sum property (that is sum of entries on rows of $A_n$ are constant for each $n\ge 0$). Clearly, when $(V,E)$ has equal path number property, there are equal number of paths from $v_0$ to each vertices in $V_n$ for $n\ge 1$. Note that the equal path number property is preserved under telescoping. We recall the following result for the sole purpose of this section. %\textcolor{blue}{add reference}.
\begin{theorem}\label{thm:toeplitz}
    A minimal subshift is Toeplitz if and only if it has a Bratteli-Vershik system $(X_B,T_B)$ where the Bratteli diagram $B=(V,E)$ satisfies equal path number property. 
\end{theorem}

By the above theorem, for a Toeplitz subshift one may take a Bratteli diagram that satisfies the equal path number property.
However, the construction of Bratteli diagram as in Lemma~\ref{lemma:A>V} to obtain a directive sequence of morphisms that satisfies Property $(P_k)$ or Property $(P_\infty)$, does not preserve the equal path number property. Hence, we have to do a different analysis in order to obtain a Toeplitz subshift with the required number of asymptotic components. We mention the construction for the case of $k$ many asymptotic components. The case for countably many asymptotic components can be easily adapted.   

Fix $k\ge 1$. Let $(V,E)$ be a simple Bratteli diagram where $V=V_0\cup V_1\cup\dots, V_n=\{v_{1,n},\dots,v_{m_n,n}\}$ and $E=E_1\cup E_2\cup\dots$. Let $A_n$ be its incidence matrix at level $n$ and $(A_n)_{i,j}$ denote the $(v_{i,n+1},v_{j,n})$-th entry of $A_n$. Telescoping if necessary, we assume that $(A_n)_{i,j}>3,$  for all $n\ge 1, 1\le i\le m_{n+1}$ and $1\le j\le m_n$.
%with respect to which the associated Bratteli-Vershik system is conjugate to $(X,T)$. 
\subsection{A directive sequence of morphisms} Given a simple Bratteli diagram, we define a partial ordering on it such that the directive sequence of morphisms read on it is given as follows. The $\mathcal{S}$-adic subshift generated by  this sequence will be shown to have $k$ asymptotic components later in this section.

\begin{definition}\label{def:almostlr}
    Let $(V,E)$ be a simple Bratteli diagram with other notations as mentioned as earlier. We define an ordered Bratteli diagram $B'=(V,E,\le')$ such that the 
   directive sequence of morphisms read on $B'$, $\BT=(\tau_n)_{n\ge 0}$ is as follows. For $n\ge 1$, $\tau_n:V^*_{n+1}\to V^*_n$ be defined as,  
\[
\tau_n(v_{i,n+1})=
\begin{cases}
v_{1,n}^2\dots v_{s-1,n}^2v_{s,n}\dots  v_{k+1,n}v_{1,n}^{\ell_{i,1}}\dots v_{m_n,n}^{\ell_{i,m_n}}, & i\equiv s\text{ mod }k+1, 1\le s\le k+1,i\ne 1\\
  v_{1,n}^{3}v_{2,n}\dots v_{k+1,n}v_{1,n}^{\ell_{1,1}}\dots v_{m_n}^{\ell_{1,m_n}},& i=1,
\end{cases}
\]
where $V_n=\{v_{1,n},\dots,v_{m_n,n}\}$, for $i\ne 1$, $\ell_{i,t}=(A_n)_{i,t}-2$ for $1\le t\le s-1$, $\ell_{i,t}=(A_n)_{i,t}-1$ for $s\le t\le k+1$, and $\ell_{i,t}=(A_n)_{i,t}$ for $k+2\le t\le m_n$, and $\ell_{1,1}=(A_n)_{1,1}-3, \ell_{1,2}=(A_n)_{1,2}-1$ and $\ell_{1,t}=(A_n)_{1,t}$ for $3\le t\le m_n$. Note that $A_n$ is the incidence matrix of $\tau_n$ for all $n\ge 1$. For $n=0,$ we define a hat morphism $\tau_0:V_1^*\to E_1^*$ as $\tau_0(v_{i,1})=e_{i,1}\dots e_{i,k_i}$
where $e_{i,j}$ is an edge from $v_{0}$ to $v_{i,1}$, $k_i=(A_0)_{i,1}$ and $V_0=\{v_0\}$. 
\end{definition}
The morphisms defined in Definition~\ref{def:almostlr} are well defined since $(A_n)_{i,j}>3$ for $n\ge 1$ and it is primitive. However, unlike morphisms that satisfy Property $(P_k)$, $\tau_n$ need not be injective on symbols. 
In any case, we have the conjugacy between the associated Bratteli-Vershik system $(X_{B'},V_{B'})$ and the $\mathcal{S}$-adic subshift $(X_{\BT},S)$ (Proposition~\ref{prop:conjugacy3}). 

\begin{remark}\label{Rem:D_s}
    For $s\in\{1,2,\dots,k+1\},$ and $n\ge 1$, denote, 
    \[
    D_{s,n}=\{v_{i,n}\mid 1<i\le m_{n},\ i\equiv s\text{ mod } k+1\} \text{ and } D_{0,n}=\{v_{1,n}\}.
    \]  
    For $1\le i\le m_{n+1}$, in $\tau_n(v_{i,n+1})$, we always have that for $j\ne k+1$, the letter $v_{j,n}$ is  followed by either $v_{j,n}$ or $v_{j+1,n}$, and $v_{k+1,n}$ is followed by either $v_{k+1,n},v_{k+2,n}$ or $v_{1,n}$. In any case, for any fixed $0\le s \le k+1$, an element of $D_{s,n}$ is never followed by another distinct element of $D_{s,n}$; the only possible exception is the element itself. In addition, elements in $V_n$ are never followed by two distinct elements belonging to the same $D_{s,n}$ for some $s$. These facts will be useful in proving the following results. 
\end{remark}

\begin{proposition}\label{prop:conjugacy3}
   Let $(V,E)$ be a simple Bratteli diagram and $B'=(V,E,\le')$ be the ordered Bratteli diagram as in Definition~\ref{def:almostlr}. Then the $\mathcal{S}$-adic subshift $(X_{\BT},S)$ is conjugate to the Bratteli-Vershik system $(X_{B'},V_{B'})$.  
\end{proposition}
\begin{proof}
    As mentioned before, let $V_n=\{v_{1,n},\dots,v_{m_n,n}\}$ for $n\ge 1$. By Proposition~\ref{prop:conjugacy}, it is enough to prove that $\tau_n$ is proper for $n\ge 1$ and that $\tau_n$ extends by concatenation to an injective map from $X_{\BT}^{(n+1)}$ to $X_{\BT}^{(n)}$. Clearly, $\tau_n$ is proper. %For $x=(x_i)_{i\in\Z}\in X_{\BT}^{(n+1)}$, we define the cutting point of $\tau_n(x)$  to be the subset of integers defined (refer to Figure~\ref{fig:cutting}) as,
    %\[
   % C_{\tau_n}(x):=\{-|\tau_n(x_{[-\ell,0)})|: \ell>0\}\cup\{0\}\cup \{|\tau_n(x_{[0,\ell)})|: \ell>0\}.
   % \]

% For a sequence $x=(x_i)_{i\in\Z}\in X_{\BT}^{(n+1)}$, by Remark~\ref{Rem:D_s}, whenever $x_i\ne x_{i+1}$, there is no $0\le s\le k+1$ such that both $x_i$ and $x_{i+1}$ belong to $D_{s,n+1}$.% That is, $s(i)\ne s(i+1)$ if and only if $x_i\ne x_{i+1}$.

\begin{figure}[ht]
\begin{center}
\begin{tikzpicture}
\beginpgfgraphicnamed{situation-b}
\begin{scope}[very thick]

\draw[loosely dotted] (-1,0) -- (0,0);
\draw (0,.15)--(0,-.15) node[above=1mm]{$x_i=v_{k,n+1}$};
\draw  (0,0) --  (10.5,0) ; 
\draw[loosely dotted] (10.5,0)--(11.5,0) node[right]{$x$};;
\draw[ dotted] (0,0)-- (-1.5,-1.5);
\draw[dotted] (0,0)-- (1.5,-1.5);
\draw (3,.15)--(3,-.15) node[above=1mm]{$x_{i+1}=v_{k,n+1}$};
\draw[ dotted] (3,0)-- (1.5,-1.5);
\draw[dotted] (3,0)-- (4.5,-1.5);
\draw (6,.15)--(6,-.15) node[above=1mm]{$x_{j-1}=v_{m_{n+1},n+1}$};
\draw[ dotted] (6,0)-- (4.5,-1.5);
\draw[dotted] (6,0)-- (7.5,-1.5);
\draw (9,.15)--(9,-.15) node[above=1mm]{$x_{j}= v_{1,n+1}$};
\draw[ dotted] (9,0)-- (7.5,-1.5);
\draw[dotted] (9,0)-- (10.5,-1.5); 

\draw[loosely dotted] (-2.5,-1.5) -- (-1.5,-1.5);
\draw[red,snake=brace] (-1.5,-1.45)-- node[above] {$\tau_n(x_i)$} (1.5,-1.45) ;
\draw[red,snake=brace, mirror snake] (-1.5,-1.55)-- node[below] {$\tau_n(x'_i)$} (1.5,-1.55) ;

\draw[red,snake=brace] (1.5,-1.45)-- node[above] {$\tau_n(x_{i+1})$} (4.5,-1.45) ;
\draw[red,snake=brace, mirror snake] (1.5,-1.55)-- node[below] {$\tau_n(x'_{i+1})$} (4.5,-1.55) ;

\draw  (-1.5,-1.5) --  (11,-1.5) ;
\draw[red,snake=brace] (7.5,-1.45)-- node[above]{$\tau_n(v_{1,n+1})$} (10.5,-1.45);
\draw[blue,snake=brace,mirror snake] (7.5,-1.55)-- node[below]{$\tau_n(x'_j)$} (10.5,-1.55);
\draw[loosely dotted] (11,-1.5)--(11.5,-1.5) node[right]{\tiny{$\tau_n(x)=\tau_n(x')$}};

\draw[loosely dotted] (-1,-3) -- (0,-3);
\draw (0,-2.85)--(0,-3.15) node[below=1mm]{$x'_i=v_{\ell,n+1}$};
\draw  (0,-3) --  (10.5,-3) ; 
\draw[loosely dotted] (10.5,-3)--(11.5,-3) node[right]{$x'$};
\draw[ dotted] (0,-3)-- (-1.5,-1.5);
\draw[dotted] (0,-3)-- (1.5,-1.5);
\draw (3,-2.85)--(3,-3.15) node[below=1mm]{$x'_i=v_{\ell,n+1}$};
\draw[ dotted] (3,-3)-- (1.5,-1.5);
\draw[dotted] (3,-3)-- (4.5,-1.5);
\draw (6,-2.85)--(6,-3.15) node[below=1mm]{$x'_{j-1}\ne v_{m_{n+1},n+1}$};
\draw[ dotted] (6,-3)-- (4.5,-1.5);
\draw[dotted] (6,-3)-- (7.5,-1.5);
\draw (9,-2.85)--(9,-3.15) node[below=1mm]{$x'_j\ne v_{1,n+1}$};
\draw[ dotted] (9,-3)-- (7.5,-1.5);
\draw[dotted] (9,-3)-- (10.5,-1.5);
\end{scope}
\endpgfgraphicnamed
\end{tikzpicture}
\caption{Image of $x$ and $x'$ under $\tau_n$ where $x\ne x'$}
\label{fig:image}
\end{center}
\end{figure}
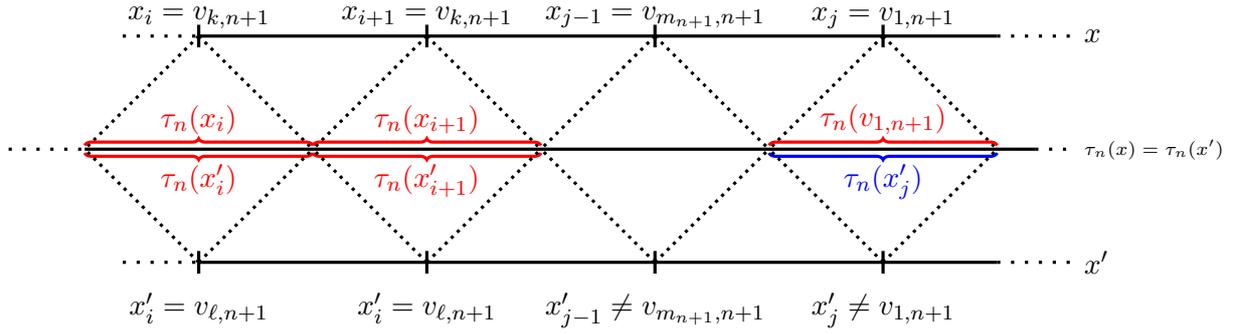

On contrary, we fix two distinct sequences $x=(x_i)_{i\in\Z},x'=(x'_i)_{i\in\Z}\in X_{\BT}^{(n+1)}$
such that $\tau_n(x)=\tau_n(x')$.    For a morphism as in Definition~\ref{def:almostlr}, the cutting points of $\tau_n(x)$ occur precisely where $v_{m_n,n}$ is followed by $v_{1,n}$.  Whenever $v_{m_n,n}$ is followed by $v_{1,n}$ in $\tau_n(x)$, the same happens in $\tau_n(x')$, and vice-versa. Hence, 
 $\tau_n(x_\ell)=\tau_n(x'_\ell)$ for all $\ell\in\Z$. 
By the definition of $\tau_n$, for each $\ell\in\Z$, there exists $0\le s(\ell)\le k+1$, such that $x_\ell,x'_\ell\in D_{s(\ell),n+1}$. By~\cref{Rem:D_s}, $s(\ell)=s(\ell+1)$ implies that $x_\ell=x_{\ell+1}$ and $x'_\ell=x'_{\ell+1}$. Hence, the change in symbols happens both in $x$ and $x'$ at the same time. 
    
Choose $i\in\mathbb{Z}$ where $x_i\ne x'_i$. Choose $m=\min_{j\ge i}\{x_jx_{j+1}=v_{m_{n+1},n+1}v_{1,n+1} \text{ or } x'_jx'_{j+1}=v_{m_{n+1},n+1}v_{1,n+1}\}$. Without loss of generality, assume that $x_mx_{m+1}=v_{m_{n+1},n+1}v_{1,n+1}$. Since $x_{m+1},x'_{m+1}\in D_{0,n+1}$, we have $x_{m+1}=x'_{m+1}=v_{1,n+1}$. Since the change of symbols in $x$ and $x'$ happen at the same time, we have $x_m=x'_m,\dots,x_i=x'_i$ which is a contradiction (Figure~\ref{fig:image}).

\end{proof}
%We say that two words $v,u$ are \emph{prefix dependent} if either $u$ is a prefix of $v$ or $v$ is a prefix of $u$. 
\noindent The following lemma is easily verifiable.
\begin{lemma}\label{lemma:prefix_indep1}
Let $\BT$ be a directive sequence of morphisms associated with a simple Bratteli diagram $(V,E)$ as given in Definition~\ref{def:almostlr}. Then for $n\ge 1$, the words $\tau_{[0,n)}(v_{i,n})$ and $\tau_{[0,n)}(v_{i+1,n})$ are not prefix dependent where $V_{n}=\{v_{1,n},\dots,v_{m_{n},n}\}$.
\end{lemma}
\begin{lemma}\label{lemma:followertop}
Let $\BT=(\tau_n)_{n\ge 0}$ be a directive sequence of morphisms associated with a simple Bratteli diagram $(V,E)$ as in Definition~\ref{def:almostlr}  with $V_n=\{v_{1,n},\dots,v_{m_n,n}\}$ and $X_{\BT}^{(n)}$ be the associated $\mathcal{S}$-adic subshift at level $n$. %For a fixed $k\ge 0$, let  $X_k:=X_{\BT}^{(k)}$ be the $\mathcal{S}$-adic shift with respect to $(\tau_n)_{n\ge k}$. 
Then, for $x^{(n)}\in X_{\BT}^{(n)}$, if $v_{i,n}\in V_n$ is a signal for $x^{(n)}$, then $1\le i\le k$. Moreover, for $y^{(n)}\in X_{\BT}^{(n)}$, and $\ell\in\Z$ where $J(x^{(n)},y^{(n)})=(v_{i,n},\ell)$, we have $\{x^{(n)}_\ell,y^{(n)}_\ell\}=\{v_{i,n},v_{i+1,n}\}$.  
    
    %Then, the only extendable right special word of length 1 in $X_k$ is $v_1$ . In particular, for any $k>0,$ there exists a word $u$ of length $k$ where $uv_1v_1,uv_1v_2\in\mathcal{L}(X)$.
\end{lemma}
\begin{proof}
% Let $A_n$ be the incidence matrix of $(V,E)$ at level $n$. 
%      Let $V_n=\{v_1,\dots,v_{m_n}\}$, $V_{n+1}=\{u_1,\dots,u_{m_{n+1}}\}$. 
    %and 
    % \[
    % M=\max_{u\in V_{n+1}}(\sum_{v\in V_n} (A_n)_{u,v})
    % \]
    % be the maximum of row sums of $A_n$. 
    If $v_{i,n}\in V_n$ is a signal for $x^{(n)}\in X_{\BT}^{(n)}$, then there exist $y^{(n)}$ and $\ell\in\Z$ such that $x^{(n)}_{(-\infty,\ell)}=y^{(n)}_{(-\infty,\ell)}$, $x^{(n)}_{\ell-1}=v_{i,n}$, and $x^{(n)}_{\ell}\ne y_{\ell}^{(n)}$. Choose a maximum of $m<\ell$ such that $x^{(n)}_{[m,\ell)}=v_{m_n,n}v_{1,n}\dots v_{m_n,n}v_{1,n}\dots v_{i,n}$. Since $X_{\BT}$ is recognizable and the cutting points of $\tau_n$ are precisely where $v_{m_n,n}$ is followed by $v_{1,n}$, we have $u$ and $u'\ne u''\in V_{n+1}$ such that $u$ is followed by $u'$ and $u''$ in $X^{(n+1)}_{\BT}$. 
    On the contrary, assume that $i>k$. In that case, $u',u''\in D_{s,n+1}$ for some $0\le s\le k+1$ which is not possible by Remark~\ref{Rem:D_s}. Clearly, for $1\le i\le k$, $v_{i,n}$ is followed by $v_{i,n}$ or $v_{i+1,n}$.
\end{proof}

\begin{remark}
By imitating the proof of Lemma~\ref{lemma:followertop}, it is not difficult to see that Lemmas~\ref{Lemma:splittingpoint},~\ref{Lemma:prefix} and~\ref{Lemma:twocomp} also hold in this setup. Hence, we can construct exactly $k$ many asymptotic components where the construction follows the proof of Theorem~\ref{thm:kmany} (it is stated below as Proposition~\ref{prop:kasymToep}). The construction of the subshift using the directive sequence of morphism as in Definition~\ref{def:almostlr} is good enough to prove Theorem~\ref{thm:kmany}. However, note that the construction using Property $(P_k)$ or $(P_\infty)$ allows us to construct infinitely many subshifts with the desired property - something that would not be possible if we were only using Definition~\ref{def:almostlr}. Hence, both constructions are significant in their own right. 
\end{remark}
\noindent We state the following result, the proof of which is omitted. 
\begin{proposition}\label{prop:kasymToep}
 Let $\BT=(\tau_n)_{n\ge 0}$ be a directive sequence of morphisms associated with a simple Bratteli diagram $(V,E)$ as in Definition~\ref{def:almostlr} and let $X_{\BT}$ be the associated $\mathcal{S}$-adic subshift. %Let $X_k$ be the $\mathcal{S}$-adic shift with respect to $(\tau_n)_{n\ge k}$ and $X_{\BT}=X_0$.
Then $X_{\BT}$ admits exactly $k$ asymptotic components. 
\end{proposition}

 \subsection{Toeplitz subshifts of trivial automorphism group}
 %We use this property to prove the following result.

\begin{proof}[Proof of Theorem~\ref{thm:main2}]
By assumption, without loss of generality, we can assume that $(X,T)$ is a Toeplitz subshift. 
    Let $B=(V,E,\ge)$ be the simple ordered Bratteli diagram with respect to which the associated Bratteli-Vershik system is conjugate to $(X,T)$. By Theorem~\ref{thm:toeplitz}, the Bratteli diagram $(V,E)$ has equal path number property. Telescoping if necessary, we make sure that we can define the directive sequence of morphisms corresponding to $(V,E)$ as in Definition~\ref{def:almostlr}. 
    Since equal path number property is preserved under telescoping, the $\mathcal{S}$-adic subshift associated with the directive sequence of morphisms as in Definition~\ref{def:almostlr} is conjugate to a Toeplitz subshift. Furthermore, it has exactly $k$ asymptotic components. 
\end{proof}
We get the following corollary by taking $k=1$.
\begin{corollary}
     Let $(X,T)$ be a Toeplitz subshift. Then there exists a Toeplitz subshift $(X',S)$ that is strong orbit equivalent to $(X,T)$ and that $\text{Aut}(X',S)$ is generated by $S$.
\end{corollary}

\begin{remark}
 Since within a strong orbit equivalence class of a Toeplitz shift, there is a Toeplitz shift of positive entropy, by Theorem~\ref{thm:main3}, we get a candidate for a Toeplitz shift having uncountably many asymptotic components. It is also not difficult to construct Toeplitz shift of countably infinite asymptotic components. However, we refrain from providing a full description to prevent redundancy. We would like to just mention the following directive sequence of morphism that works. For $n\ge 1$, we define $\tau_n:V_{n+1}^*\to V_n^*$ as, 
 \[
\tau_n(v_{i,n+1})=
\begin{cases}
v_{1,n}^2\dots v_{s-1,n}^2v_{s,n}\dots v_{n+1,n}v_{1,n}^{\ell_{i,1}}\dots v_{m_n,n}^{\ell_{i,m_n}}, & i\equiv s\text{ mod }n+1, 1\le s\le n+1,i\ne 1\\
  v_{1,n}^{3}v_{2,n}\dots v_{n+1,n}v_{1,n}^{\ell_{1,1}}\dots v_{m_n,n}^{\ell_{1,m_n}},& i=1,
\end{cases}
\] where $V_n=\{v_{1,n},\dots,v_{m_n,n}\}$, $\tau_0$ is a hat morphism and $\ell_{i,j}$'s are chosen accordingly so that the newly constructed subshift is strong orbit equivalent to the Toeplitz shift that we start with. 
\end{remark}

% \subsection{Things to do}
% \begin{enumerate}
%     \item When are two Bratteli systems having the same diagram but different ordering non-conjugate?
%     \item Can we construct subshifts of higher complexity?
% \end{enumerate}
% \addcontentsline{toc}{chapter}{References}
	\bibliographystyle{abbrv}   %bibliography style
	\renewcommand{\bibname}{References} %bibliography chapter header
	\bibliography{mybib} %.bib file source

\end{document}